\documentclass[a4paper,12pt]{article}

\usepackage{amsmath,amssymb,amsthm,amscd,color}

\allowdisplaybreaks

\newcommand{\Fg}{\mathfrak{g}}
\newcommand{\Fh}{\mathfrak{h}}

\newcommand{\BC}{\mathbb{C}}
\newcommand{\BR}{\mathbb{R}}
\newcommand{\BQ}{\mathbb{Q}}
\newcommand{\BZ}{\mathbb{Z}}
\newcommand{\BB}{\mathbb{B}}

\newcommand{\Hom}{\mathop{\rm Hom}\nolimits}
\newcommand{\GL}{\mathop{\rm GL}\nolimits}
\newcommand{\wt}{\mathop{\rm wt}\nolimits}
\newcommand{\cl}{\mathop{\rm cl}\nolimits}
\newcommand{\dist}{\mathop{\rm dist}\nolimits}
\newcommand{\Deg}{\mathop{\rm Deg}\nolimits}
\newcommand{\ext}{\mathop{\rm ext}\nolimits}

\newcommand{\rr}{\Delta_{\mathrm{re}}}
\newcommand{\prr}{\Delta_{\mathrm{re}}^{+}}

\newcommand{\pair}[2]{\langle #1,\,#2 \rangle}
\newcommand{\Bpair}[2]{\bigl\langle #1,\,#2 \bigr\rangle}
\newcommand{\len}[2]{\ell(#2,\,#1)}
\newcommand{\mcr}[1]{\lfloor #1 \rfloor}

\newcommand{\vpi}{\varpi}

\newcommand{\ha}[1]{\widehat{#1}}
\newcommand{\ti}[1]{\widetilde{#1}}
\newcommand{\ud}[1]{\underline{#1}}

\newcommand{\bzero}{{\bf 0}}
\newcommand{\bd}{{\bf d}}

\newcommand{\bi}{{\bf i}}

\newcommand{\Q}{{\mathsf q}}

\makeatletter
\@addtoreset{equation}{subsection}

\renewcommand\section{\@startsection{section}{1}{0pt}
{-3.5ex plus -1ex minus -.2ex}{1.0ex plus .2ex}{\large\bf}}
\renewcommand\subsection{\@startsection{subsection}{1}{0pt}
{2.5ex plus 1ex minus .2ex}{-1em}{\bf}}
\makeatother

\theoremstyle{plain}
\newtheorem{thm}{Theorem}[subsection]
\newtheorem{lem}[thm]{Lemma}
\newtheorem{prop}[thm]{Proposition}

\newtheorem*{claim*}{Claim}

\theoremstyle{definition}
\newtheorem{dfn}[thm]{Definition}

\theoremstyle{remark}
\newtheorem{rem}[thm]{Remark}
\newtheorem{ex}[thm]{Example}

\newenvironment{enu}{%
 \begin{enumerate}%
 \renewcommand{\labelenumi}{\rm (\theenumi)}%
}{\end{enumerate}}

\begin{document}

\setlength{\baselineskip}{18pt}

\title{\Large\bf Explicit description of \\[2mm]
the degree function in terms of \\[2mm]
quantum Lakshmibai-Seshadri paths}
\author{
 Cristian Lenart \\ 
 \small Department of Mathematics and Statistics, 
 State University of New York at Albany, \\ 
 \small Albany, NY 12222, U.\,S.\,A. \ 
 (e-mail: {\tt clenart@albany.edu}) \\[5mm]
 Satoshi Naito \\ 
 \small Department of Mathematics, Tokyo Institute of Technology, \\
 \small 2-12-1 Oh-okayama, Meguro-ku, Tokyo 152-8551, Japan \ 
 (e-mail: {\tt naito@math.titech.ac.jp}) \\[5mm]
 Daisuke Sagaki \\ 
 \small Institute of Mathematics, University of Tsukuba, \\
 \small Tsukuba, Ibaraki 305-8571, Japan \ 
 (e-mail: {\tt sagaki@math.tsukuba.ac.jp}) \\[5mm]
 Anne Schilling \\ 
 \small Department of Mathematics, University of California, \\
 \small One Shields Avenue, Davis, CA 95616-8633, U.\,S.\,A. \ 
 (e-mail: {\tt anne@math.ucdavis.edu}) \\[5mm]
 Mark Shimozono \\ 
 \small Department of Mathematics, MC 0151, 460 McBryde Hall, 
        Virginia Tech, \\
 \small 225 Stanger St., Blacksburg, VA 24061, 
 U.\,S.\,A. \ 
 (e-mail: {\tt mshimo@vt.edu})
}
\date{}
\maketitle

%
\begin{abstract} \setlength{\baselineskip}{16pt}
We give an explicit and computable description, in terms of 
the parabolic quantum Bruhat graph, of the degree function 
defined for quantum Lakshmibai-Seshadri paths, or equivalently, 
for ``projected'' (affine) level-zero Lakshmibai-Seshadri paths. 
This, in turn, gives an explicit and computable description of 
the global energy function on tensor products of 
Kirillov-Reshetikhin crystals of one-column type, and also of 
(classically restricted) one-dimensional sums.
\end{abstract}
%
%
\section{Introduction.}
\label{sec:intro}

Let $\Fg$ be an affine Lie algebra with index set $I$ for the simple roots, and 
let $U_{q}'(\Fg)$ be the quantum affine algebra (without the degree operator) 
associated to $\Fg$. Set $I_{0}:=I \setminus \{0\}$, 
where $0 \in I$ corresponds to the ``extended'' vertex in the Dynkin diagram of $\Fg$.
In \cite{NS-IMRN, NS-Tensor, NS-Adv}, 
Naito and Sagaki gave a combinatorial realization of the crystal bases of 
tensor products of level-zero fundamental representations 
$W(\varpi_{i})$, $i \in I_{0}$, over 
$U_{q}^{\prime}(\mathfrak{g})$, where the $\varpi_{i}$'s are 
the level-zero fundamental weights; 
the $U_{q}^{\prime}(\mathfrak{g})$-modules $W(\varpi_{i})$ are often called 
Kirillov-Reshetikhin (KR for short) modules of one-column type, 
and accordingly their crystal bases are called KR crystals of one-column type.
In the papers above, they realized elements of the crystal bases 
as projected (affine) level-zero Lakshmibai-Seshadri (LS for short) paths.
Here a projected level-zero LS path is obtained from an ordinary LS path 
of shape $\lambda$ by factoring out the null root $\delta$ of 
the affine Lie algebra $\mathfrak{g}$, where $\lambda$ is 
a level-zero dominant integral weight of the form 
$\lambda = \sum_{i \in I_{0}} m_{i} \varpi_{i}$, 
with $m_{i} \in \mathbb{Z}_{\geq 0}$.
However, from the nature of the definition above of 
projected level-zero LS paths, their description of these objects in 
\cite{NS-IMRN, NS-Tensor, NS-Adv} is not as explicit as the one of 
usual LS paths given by Littelmann in \cite{Lit-I}. 

By contrast, in our previous paper \cite{LNSSS3}, 
we proved that (in the case that $\Fg$ is an untwisted affine Lie algebra)
a projected level-zero LS path is identical 
to what we call a quantum LS path, which is described quite explicitly 
in terms of the parabolic quantum Bruhat graph, instead of 
(the Hasse diagram of) the usual Bruhat graph.

Also, in \cite{NS-Degree}, 
we defined a certain integer-valued function, 
called the degree function, on the set $\mathbb{B}(\lambda)_{\cl}$ 
of projected level-zero LS paths of shape 
$\lambda = \sum_{i \in I_{0}} m_{i} \varpi_{i}$, and 
proved that it is identical to the global ``energy function'' 
on the tensor product $\bigotimes_{i \in I_{0}} 
\mathbb{B}(\varpi_{i})_{\cl}^{\otimes m_{i}}$ 
under the isomorphism $\mathbb{B}(\lambda)_{\cl} \cong 
\bigotimes_{i \in I_{0}} \mathbb{B}(\varpi_{i})_{\cl}^{\otimes m_{i}}$ 
of $U_{q}^{\prime}(\mathfrak{g})$-crystals; recall that 
for each $i \in I_{0}$, the crystal 
$\mathbb{B}(\varpi_{i})_{\cl}$ is isomorphic, 
as a $U_{q}^{\prime}(\mathfrak{g})$-crystal, 
to a KR crystal of one-column type.
However, again from the nature of the definition of 
projected level-zero LS paths, our description in \cite{NS-Degree} 
is not very explicit, and hence it is difficult to compute 
the value of the degree function at 
a given projected level-zero LS path.

In \cite{LNSSS2}, we give an explicit and computable description, 
in terms of the parabolic quantum Bruhat graph, 
of the degree function defined for quantum LS paths, 
or equivalently, for projected level-zero LS paths \cite{LNSSS3}. 
This, in turn, gives a new description of 
the global energy function on tensor products of KR crystals of 
one-column type, and also of (classically restricted) one-dimensional 
sums arising from the study of solvable lattice models 
in statistical mechanics through Baxter's corner transfer matrix method 
(for details, see \cite{S}).

The purpose of this paper is to give a new proof of 
the description above, in terms of the parabolic quantum Bruhat graph, 
of the degree function. We should mention that our proof in this paper 
is completely different from the one in \cite{LNSSS2} in that 
(at least in appearance) we do not make use of root operators; 
it is based on a technical lemma (Lemma~\ref{lem:C}) about the decomposition of 
$\BB(\lambda)$ into connected components, and also on our results 
in \cite{LNSSS1}, where $\BB(\lambda)$ denotes the crystal of 
(not projected) LS paths of shape $\lambda$. 

This paper is organized as follows. In \S\ref{sec:LS}, 
we fix our basic notation, and review some fundamental facts about
level-zero path crystals. Also, we recall the definition of 
the degree function, and then prove a technical lemma (Lemma~\ref{lem:C}), 
which plays an important rule in the proof of our main result (Theorem~\ref{thm:main}). 
In \S\ref{sec:QLS}, we recall the notion of 
parabolic quantum Bruhat graph, and then give 
the definition of quantum LS paths. 
In \S\ref{sec:main}, we state and prove our main result about 
the description of the degree function in terms of 
the parabolic quantum Bruhat graph. 

\paragraph{Acknowledgments.}
C.L. was partially supported by the NSF grant DMS--1101264. 
S.N. was supported by Grant-in-Aid for Scientific Research (C), No.\,24540010, Japan. 
D.S. was supported by Grant-in-Aid for Young Scientists (B), No.\,23740003, Japan.
A.S. was partially supported by the NSF grant OCI--1147247.
M.S. was partially supported by the NSF grant DMS--1200804.

%
\section{Lakshmibai-Seshadri paths and the degree function.}
\label{sec:LS}

%
\subsection{Basic notation.}
\label{subsec:notation}

Let $\Fg$ be an untwisted affine Lie algebra 
over $\BC$ with Cartan matrix $A=(a_{ij})_{i,\,j \in I}$; 
throughout this paper, the elements of the index set $I$ 
are numbered as in \cite[\S4.8, Table Aff~1]{Kac}. 
Take a distinguished vertex $0 \in I$ as in \cite{Kac}, 
and set $I_{0}:=I \setminus \{0\}$.
Let 
$\Fh=\bigl(\bigoplus_{j \in I} \BC \alpha_{j}^{\vee}\bigr) \oplus \BC d$
denote the Cartan subalgebra of $\Fg$, where 
$\Pi^{\vee}:=\bigl\{\alpha_{j}^{\vee}\bigr\}_{j \in I} \subset \Fh$ is 
the set of simple coroots, and 
$d \in \Fh$ is the scaling element (or degree operator). 
We denote by $\pair{\cdot\,}{\cdot}:
\Fh^{\ast} \times \Fh \rightarrow \BC$ 
the duality pairing between $\Fh^{\ast}:=\Hom_{\BC}(\Fh,\,\BC)$ and $\Fh$. 
Denote by $\Pi:=\bigl\{\alpha_{j}\bigr\}_{j \in I} \subset 
\Fh^{\ast}:=\Hom_{\BC}(\Fh,\BC)$ 
the set of simple roots, and by 
$\Lambda_{j} \in \Fh^{\ast}$, $j \in I$, 
the fundamental weights; 
note that $\pair{\alpha_{j}}{d}=\delta_{j,0}$ and 
$\pair{\Lambda_{j}}{d}=0$ for $j \in I$.
Let $\delta=\sum_{j \in I} a_{j}\alpha_{j} \in \Fh^{\ast}$ and 
$c=\sum_{j \in I} a^{\vee}_{j} \alpha_{j}^{\vee} \in \Fh$ denote 
the null root and the canonical central element of 
$\Fg$, respectively. 
The Weyl group $W$ of $\Fg$ is defined as
$W:=\langle r_{j} \mid j \in I\rangle \subset \GL(\Fh^{\ast})$, 
where $r_{j} \in \GL(\Fh^{\ast})$ denotes the simple reflection 
associated to $\alpha_{j}$ for $j \in I$, 
with $\ell:W \rightarrow \BZ_{\ge 0}$ the length function on $W$. 
Denote by $\rr$ the set of real roots, i.e., $\rr:=W\Pi$, 
and by $\prr \subset \rr$ the set of positive real roots; 
for $\beta \in \rr$, we denote by $\beta^{\vee}$ 
the dual root of $\beta$, and by $r_{\beta} \in W$ 
the reflection associated to $\beta$. 
We take a dual weight lattice $P^{\vee}$ 
and a weight lattice $P$ as follows:
%
%
\begin{equation} \label{eq:lattices}
P^{\vee}=
\left(\bigoplus_{j \in I} \BZ \alpha_{j}^{\vee}\right) \oplus \BZ d \, 
\subset \Fh
\quad \text{and} \quad 
P= 
\left(\bigoplus_{j \in I} \BZ \Lambda_{j}\right) \oplus 
   \BZ \delta \subset \Fh^{\ast}.
\end{equation}
It is clear that $P$ contains the root lattice
$Q:=\bigoplus_{j \in I} \BZ \alpha_{j}$, and that 
$P \cong \Hom_{\BZ}(P^{\vee},\BZ)$. 

Let $W_{0}$ denote the subgroup of $W$ generated by 
$r_{j}$, $j \in I_{0}$. 
Set $Q_{0}:=\bigoplus_{j \in I_{0}}\BZ\alpha_{j}$, 
$Q_{0}^{+}:=\sum_{j \in I_{0}}\BZ_{\ge 0}\alpha_{j}$, 
$\Delta_{0}:=\rr \cap Q_{0}$, $\Delta_{0}^{+}:=
\rr \cap Q_{0}^{+}$, and
$\Delta_{0}^{-}:=-\Delta_{0}^{+}$. 
Note that $W_{0}$ (resp., $\Delta_{0}$, $\Delta_{0}^{+}$, $\Delta_{0}^{-}$) 
can be thought of as the (finite) Weyl group 
(resp., the set of roots, the set of positive roots, the set of negative roots) 
of the finite-dimensional simple Lie subalgebra of $\Fg$ 
corresponding to the subset $I_{0}$ of $I$. 
Also, we denote by $\theta \in \Delta_{0}^{+}$ the highest root of 
the (finite) root system $\Delta_{0}$; note that 
$\alpha_{0}=-\theta+\delta$ and $\alpha_{0}^{\vee}=-\theta^{\vee}+c$. 
%
%
%
\begin{dfn} \label{dfn:lv0}
\mbox{} \samepage
\begin{enu}
\item
An integral weight 
$\lambda \in P$ is said to 
be of level zero if $\pair{\lambda}{c}=0$. 

\item
An integral weight 
$\lambda \in P$ is said to 
be level-zero dominant if $\pair{\lambda}{c}=0$, and 
$\pair{\lambda}{\alpha_{j}^{\vee}} \ge 0$ 
for all $j \in I_{0}=I \setminus \{0\}$. 
\end{enu}
\end{dfn}
%
%
\begin{rem} \label{rem:theta}
\mbox{}
\begin{enu}
\item If $\lambda \in P$ is of level zero, 
then $\pair{\lambda}{\alpha_{0}^{\vee}}=-\pair{\lambda}{\theta^{\vee}}$.

\item For $h \in Q^{\vee}_{0}:=\bigoplus_{j \in I_{0}} \BZ \alpha_{j}^{\vee}$, 
we denote by $t_{h} \in W$ the translation with respect to $h$ (see \cite[\S6.5]{Kac}). 
If $\lambda$ is of level-zero, then 
$t_{h}\lambda=\lambda-\pair{\lambda}{h}\delta$ for $h \in Q_{0}^{\vee}$. 
Because $W$ is the semidirect product of $W_{0}$ and the abelian (normal) subgroup 
$T=\bigl\{t_{h} \mid h \in Q^{\vee}_{0}\bigr\}$ of translations by \cite[Proposition~6.5]{Kac}, 
we deduce (see also \cite[Lemma~2.6]{NS-LMS} for example) that 
if $\lambda$ is level-zero dominant, then $W\lambda=W_{0}T\lambda \subset 
W_{0}\lambda+\BZ\delta \subset \lambda-Q_{0}^{+}+\BZ\delta$; 
we define $d_{\lambda} \in \BZ_{> 0}$ by: 
$\bigl\{n \in \BZ \mid \lambda+n\delta \in T\lambda \bigr\}=d_{\lambda}\BZ$.
\end{enu}
\end{rem}

For each $i \in I_{0}$, 
we define a level-zero fundamental weight 
$\vpi_{i} \in P$ by
\begin{equation}
\vpi_{i}:=\Lambda_{i}-a_{i}^{\vee}\Lambda_{0}.
\end{equation}
The weights $\vpi_{i}$ for $i \in I_{0}$ 
are actually level-zero dominant integral weights; 
indeed, $\pair{\vpi_{i}}{c}=0$ and 
$\pair{\vpi_{i}}{\alpha_{j}^{\vee}}=\delta_{i,j}$ for $i,\,j \in I_{0}$.

Let $\cl:\Fh^{\ast} \twoheadrightarrow \Fh^{\ast}/\BC\delta$ 
denote the canonical projection from $\Fh^{\ast}$ onto 
$\Fh^{\ast}/\BC\delta$, and 
define $P_{\cl}$ and $P_{\cl}^{\vee}$ by
%
%
\begin{equation} \label{eq:lat-cl}
P_{\cl} := \cl(P) = 
 \bigoplus_{j \in I} \BZ \cl(\Lambda_{j})
\quad \text{and} \quad 
P_{\cl}^{\vee} := 
 \bigoplus_{j \in I} \BZ \alpha_{j}^{\vee} 
 \subset P^{\vee}.
\end{equation}
We see that 
$P_{\cl} \cong P/\BZ\delta$, and that 
$P_{\cl}$ can be identified with 
$\Hom_{\BZ}(P_{\cl}^{\vee},\BZ)$ as a $\BZ$-module by
%
%
\begin{equation} \label{eq:pair}
\pair{\cl(\lambda)}{h}=\pair{\lambda}{h} \quad
\text{for $\lambda \in P$ and $h \in P_{\cl}^{\vee}$}.
\end{equation}
Also, there exists 
a natural action of the Weyl group $W$ on 
$\Fh^{\ast}/\BC\delta$ induced by the one on $\Fh^{\ast}$, 
since $W\delta=\delta$; it is obvious that
$w \circ \cl = \cl \circ w$ for all $w \in W$.
%
%
\begin{rem} \label{rem:orbcl}
Let $\lambda \in P$ be a level-zero integral weight. 
It is easy to check that 
$\cl(W\lambda)=W_{0}\cl(\lambda)$ 
(see the proof of \cite[Lemma~2.3.3]{NS-LMS}). 
In particular, we have 
$\cl(r_{0}\lambda)=r_{\theta}\lambda$ 
since $\alpha_{0}=-\theta+\delta$ and 
$\alpha_{0}^{\vee}=-\theta^{\vee}+c$. 
\end{rem}

For simplicity of notation, 
we often write $\beta$ instead of $\cl(\beta) \in P_{\cl}$ 
for $\beta \in \bigoplus_{j \in I} \BZ \alpha_{j}$; 
note that $\alpha_{0}=-\theta$ in $P_{\cl}$ 
since $\alpha_{0}=-\theta+\delta$ in $P$. 

%
\subsection{Lakshmibai-Seshadri paths.}
\label{subsec:LS}

Here we recall the definition of 
Lakshmibai-Seshadri (LS for short) paths
from \cite[\S4]{Lit-A}. In this subsection, 
we fix a level-zero dominant integral 
weight $\lambda \in \sum_{i \in I_{0}}\BZ_{\ge 0}\vpi_{i}$. 
%
%
\begin{dfn} \label{dfn:Bruhat}
For $\mu,\,\nu \in W\lambda$, 
let us write $\mu \ge \nu$ if there exists a sequence 
$\mu=\mu_{0},\,\mu_{1},\,\dots,\,\mu_{n}=\nu$ 
of elements in $W\lambda$ and a sequence 
$\xi_{1},\,\dots,\,\xi_{n} \in \prr$ of 
positive real roots such that
$\mu_{k}=r_{\xi_{k}}(\mu_{k-1})$ and 
$\pair{\mu_{k-1}}{\xi^{\vee}_{k}} < 0$ 
for $k=1,\,2,\,\dots,\,n$. 
If $\mu \ge \nu$, then we define $\dist(\mu,\nu)$ to 
be the maximal length $n$ of all possible such sequences 
$\mu_{0},\,\mu_{1},\,\dots,\,\mu_{n}$ for $(\mu,\nu)$.
\end{dfn}
%
%
\begin{rem} \label{rem:LP}
Keep the notation of Definition~\ref{dfn:Bruhat}. 
We see that 
\begin{equation*}
\nu-\mu = \sum_{k=1}^{n} (\mu_{k}-\mu_{k-1}) = 
-\sum_{k=1}^{n}\underbrace{\pair{\mu_{k-1}}{\xi_{k}^{\vee}}}_{< 0}\xi_{k} 
\in \sum_{j \in I}\BZ_{\ge 0} \alpha_{j}.
\end{equation*}
\end{rem}

It is obvious that $\mu$ covers $\nu$ in the poset $W\lambda$ 
if and only if $\mu > \nu$ with $\dist(\mu,\,\nu)=1$. 
In this case, we write $\mu \gtrdot \nu$. 
%
%
\begin{rem} \label{rem:cover}
Let $\mu,\,\nu \in W\lambda$ be such that $\mu \gtrdot \nu$, 
and let $\xi \in \prr$ be the positive real root such that 
$r_{\xi}\mu=\nu$. We know from \cite[Lemma 2.11]{NS-LMS} that 
$\xi \in \Delta_{0}^{+} \sqcup \bigl\{-\gamma+\delta \mid 
\gamma \in \Delta_{0}^{+}\bigr\}$. 
\end{rem}
%
%
\begin{dfn} \label{dfn:achain}
For $\mu,\,\nu \in W\lambda$ with $\mu > \nu$ and 
a rational number $0 < \sigma < 1$, 
a $\sigma$-chain for $(\mu,\nu)$ is, 
by definition, a sequence $\mu=\mu_{0} \gtrdot \mu_{1} \gtrdot 
\dots \gtrdot \mu_{n}=\nu$ of elements in $W\lambda$ 
such that $\sigma\pair{\mu_{k-1}}{\xi_{k}^{\vee}} \in \BZ_{< 0}$ 
for all $k=1,\,2,\,\dots,\,n$, 
where $\xi_{k}$ is the positive real root such that 
$r_{\xi_{k}}\mu_{k-1}=\mu_{k}$. 
%
%
\end{dfn}
%
%
\begin{dfn} \label{dfn:LS}
An LS path of shape $\lambda$ is, by definition, 
a pair $(\ud{\nu}\,;\,\ud{\sigma})$ of a sequence 
$\ud{\nu}:\nu_{1} > \nu_{2} > \cdots > \nu_{s}$ of 
elements in $W\lambda$ and a sequence 
$\ud{\sigma}:0=\sigma_{0} < \sigma_{1} < \cdots < \sigma_{s}=1$ of 
rational numbers satisfying the condition that 
there exists a $\sigma_{k}$-chain for $(\nu_{k},\,\nu_{k+1})$ 
for each $k=1,\,2,\,\dots,\,s-1$. We denote by $\BB(\lambda)$ 
the set of all LS paths of shape $\lambda$.
\end{dfn}

We identify $\pi=(\nu_{1},\,\nu_{2},\,\dots,\,\nu_{s}\,;\,
\sigma_{0},\,\sigma_{1},\,\dots,\,\sigma_{s}) \in \BB(\lambda)$ with 
the following piecewise-linear, continuous map 
$\pi:[0,1] \rightarrow \BR \otimes_{\BZ} P$: 
%
%
\begin{equation} \label{eq:path}
\pi(t)=\sum_{l=1}^{k-1}
(\sigma_{l}-\sigma_{l-1})\nu_{l}+
(t-\sigma_{k-1})\nu_{k} \quad 
\text{for $\sigma_{k-1} \le t \le \sigma_{k}$, $1 \le k \le s$}. 
\end{equation}
%
%
\begin{rem} \label{rem:str01}
It is obvious from the definition that for every $\nu \in W\lambda$, 
$\pi_{\nu}:=(\nu\,;\,0,1)$ is an LS path of shape $\lambda$, 
which corresponds (under \eqref{eq:path}) to the straight line path
$\pi_{\nu}(t)=t\nu$, $t \in [0,1]$, connecting $0$ to $\nu$. 
\end{rem}

For $\pi \in \BB(\lambda)$, we define $\cl(\pi):[0,1] \rightarrow 
\BR \otimes_{\BZ} P_{\cl}$ by
\begin{equation*}
\cl(\pi)(t):=\cl(\pi(t)) \quad \text{for $t \in [0,1]$}.
\end{equation*}
Also, we set
\begin{equation*} 
\BB(\lambda)_{\cl}:=\cl(\BB(\lambda))=
\bigl\{\cl(\pi) \mid \pi \in \BB(\lambda)\bigr\}. 
\end{equation*}

%
\begin{rem} \label{rem:str02}
For $\mu \in P_{\cl}$, we define $\eta_{\mu}(t):=t\mu$ 
for $t \in [0,1]$. It is easily seen 
from Remark~\ref{rem:str01} that 
$\eta_{\mu}$ is contained in $\BB(\lambda)_{\cl}$ 
for all $\mu \in \cl(W\lambda)=W_{0}\cl(\lambda)$. 
\end{rem}

We can endow the set $\BB(\lambda)$ of LS paths of shape $\lambda$
(resp., the set $\BB(\lambda)_{\cl}$ of ``$\cl$-projected'' LS paths of shape $\lambda$) 
with a crystal structure with weights in $P$ (resp., in $P_{\cl}$) by defining 
root operators on $\BB(\lambda)$ (resp., $\BB(\lambda)_{\cl}$); since we do not 
use root operators in this paper, we omit the details (see \cite{Lit-A}, 
and also \cite[\S2.2]{NS-Degree}, \cite[\S2.3]{LNSSS2}).

%
\subsection{Degree function.}
\label{subsec:degree}

As in the previous subsection, we fix a level-zero dominant integral weight 
$\lambda \in \sum_{i \in I_{0}}\BZ_{\ge 0}\vpi_{i}$. 
Let us recall the definition of the degree function 
$\Deg=\Deg_{\lambda}:\BB(\lambda)_{\cl} \rightarrow \BZ_{\le 0}$
from \cite[\S3.1]{NS-Degree}. We know the following 
proposition from \cite[Proposition~3.1.3]{NS-Degree}.
%
%
\begin{prop} \label{prop:degree}
Let $\lambda \in \sum_{i \in I_{0}}\BZ_{\ge 0}\vpi_{i}$ be 
a level-zero dominant integral weight. For each 
$\eta \in \BB(\lambda)_{\cl}$, there exists a unique 
element $\pi_{\eta} \in \BB(\lambda)$ satisfying 
the following conditions: 
\begin{enu}
\item $\cl(\pi_{\eta})=\eta$; 

\item the element $\pi_{\eta}$ is contained in the connected component 
$\BB_{0}(\lambda)$ of $\BB(\lambda)$ containing the straight line path
$\pi_{\lambda}=(\lambda\,;\,0,\,1) \in \BB(\lambda)$; 

\item if we write $\pi_{\eta}$ in the form 
$(\nu_{1},\,\nu_{2},\,\dots,\,\nu_{s}\,;\,\ud{\sigma})$ as in Definition~\ref{dfn:LS}, 
then $\nu_{1}$ is contained in the set $\lambda-Q_{0}^{+}$ 
{\rm(}see Remark~\ref{rem:theta}\,(2){\rm)}. 
\end{enu}
\end{prop}

Let $\eta \in \BB(\lambda)_{\cl}$. 
It follows from \cite[Lemma~3.1.1]{NS-Degree} that 
$\pi_{\eta}(1) \in P$ is of the form 
$\pi_{\eta}(1)=\lambda-\beta+K\delta$ 
for some $\beta \in Q_{0}^{+}$ and $K \in \BZ_{\ge 0}$. 
We define 
\begin{equation*}
\Deg(\eta)=\Deg_{\lambda}(\eta):=-K \in \BZ_{\le 0}.
\end{equation*}

The following lemma plays an important role 
in the proof of Theorem~\ref{thm:main}.
%
%
\begin{lem} \label{lem:C}
Let $C$ be a connected component of $\BB(\lambda)$. 
%
\begin{enu}

\item For each $\eta \in \BB(\lambda)_{\cl}$, 
there exists a unique element $\pi^{C}_{\eta} \in C$ 
satisfying the same conditions as (1) and (3) of
Proposition~\ref{prop:degree}. 

\item If $\pi_{\eta}(1)=\lambda-\beta-\Deg(\eta)\delta$
with $\beta \in Q_{0}^{+}$, then 
$\pi^{C}_{\eta}(1)=\lambda-\beta+\bigl(-\Deg(\eta)+L\bigr)\delta$ 
for some $L \in \BZ_{\ge 0}$. 

\item In part (2) above, 
$C=\BB_{0}(\lambda)$ if and only if $L=0$. 
\end{enu}
\end{lem}

\begin{proof}
If $C = \BB_{0}(\lambda)$, then we have $\pi_{\eta}^{C} =
\pi_{\eta}$. In this case, part (1) follows from 
Proposition~\ref{prop:degree}; part (2) and 
the ``only if'' part of part (3) are obvious.

Assume that $C \ne \BB_{0}(\lambda)$.  
We see from \cite[Theorem~3.1 and Remark~2.15]{NS-LMS} that 
the connected component
$C$ contains a unique element $\pi_{\lambda}^{C}$ of the form 
%
%
\begin{equation} \label{eq:NS4}
\pi_{\lambda}^{C}
 =(\lambda-N_{1}\delta, \, \dots, \, \lambda-N_{s-1}\delta, \, \lambda \,;\, 
   \tau_{0},\,\tau_{1},\,\dots,\,\tau_{s-1},\,\tau_{s})
\end{equation}
for some integers $N_{1} > N_{2} > \cdots > N_{s-1} > N_{s}=0$ and 
rational numbers $0=\tau_{0} < \tau_{1} < \cdots < \tau_{s}=1$; 
since $C \ne \BB_{0}(\lambda)$ (and hence $\pi_{\lambda}^{C} \ne \pi_{\lambda}$), 
we have $s > 1$. From \eqref{eq:NS4}, by using \eqref{eq:path}, 
we deduce that
\begin{equation*}
\pi_{\lambda}^{C}(1) = 
  \lambda - \Biggl(
    \underbrace{\sum_{u=1}^{s} (\tau_{u}-\tau_{u-1}) N_{u}}_{=:N}\Biggr)\delta; 
\end{equation*}
note that $N \in \mathbb{Z}$ since $\pi_{\lambda}^{C}(1) \in P$,
which in turn follows from the integrality condition on LS paths
(see Definitions~\ref{dfn:achain} and \ref{dfn:LS}). 
Also, since $N_{1} > N_{2} > \cdots > N_{s-1} > N_{s}=0$ with $s > 1$, 
it follows that 
\begin{equation*}
N=\sum_{u=1}^{s} (\tau_{u}-\tau_{u-1}) N_{u} < 
\sum_{u=1}^{s} (\tau_{u}-\tau_{u-1}) N_{1} = N_{1}. 
\end{equation*}
Therefore, we have $\pi_{\lambda}^{C}(1) = 
\lambda-N_{1}\delta+L\delta$, with $L:=N_{1}-N \in \BZ_{> 0}$. 

Let us denote by $F:[0,1] \rightarrow \BR \otimes_{\BZ} P$ 
the piecewise-linear, continuous function such that 
$\pi_{\lambda}^{C}(t)=\pi_{\lambda}(t)+F(t)\delta$ 
for all $t \in [0,1]$; note that $F(0)=0$, 
%
%
\begin{equation} \label{eq:M}
\lim_{ \begin{subarray}{c} t \to 0 \\ t > 0 \end{subarray} } \frac{F(t)-F(0)}{t-0} 
= \lim_{ \begin{subarray}{c} t \to 0 \\ t > 0 \end{subarray} } \frac{F(t)}{t} = -N_{1},
\end{equation}
and $F(1)=-N_{1}+L$. Then, by using \cite[Lemma 2.26]{NS-LMS}, we deduce that 
$C=\bigl\{\pi(t)+F(t)\delta \mid \pi \in \BB_{0}(\lambda)\bigr\}$. 
Hence it follows from \cite[Lemma 3.1.2]{NS-Degree} that 
\begin{equation*}
\bigl\{\pi \in C \mid \cl(\pi)=\eta\bigr\} = 
\bigl\{\pi_{\eta}(t)+t(M\delta)+F(t)\delta \mid 
 M \in d_{\lambda}\BZ\bigr\};
\end{equation*}
recall the notation $d_{\lambda} \in \BZ_{\ge 0}$ 
from Remark~\ref{rem:theta}\,(2).
Therefore, we conclude 
by Proposition~\ref{prop:degree} and \eqref{eq:M} 
that $\pi_{\eta}^{C}(t):=
\pi_{\eta}(t)+F(t)\delta+t(N_{1}\delta)$ is a unique 
element in $C$ satisfying the same conditions as (1) and (3) 
of Proposition~\ref{prop:degree}. 
This proves part (1) for $C \ne \BB_{0}(\lambda)$. 
Moreover, part (2) for $C \ne \BB_{0}(\lambda)$ and 
the ``if'' part of part (3) follow immediately since  
\begin{equation*}
\pi_{\eta}^{C}(1)=\pi_{\eta}(1)+F(1)\delta+N_{1}\delta = 
\pi_{\eta}(1) + L\delta
\end{equation*}
with $L > 0$. 
This completes the proof of the lemma.
\end{proof}

%
\subsection{Global energy function.}
\label{subsec:comments}

We know from \cite[Proposition~5.8]{NS-IMRN} and 
\cite[Theorem 2.1.1 and Proposition~3.4.2]{NS-Adv} 
that for each $i \in I_{0}$, the crystal $\BB(\vpi_{i})_{\cl}$ 
is isomorphic, as a crystal with weights in $P_{\cl}$, 
to the crystal basis of the level-zero fundamental 
representation $W(\vpi_{i})$ introduced in \cite[Theorem~5.17]{Kas-OnL}; 
the level-zero fundamental modules $W(\varpi_{i})$, $i \in I_{0}$, are 
often called Kirillov-Reshetikhin (KR for short) modules of 
one-column type, and accordingly their crystal bases are called 
KR crystals of one-column type. 
Also, we know the following from \cite[Theorem 3.2]{NS-Tensor}. 
Let $\bi=(i_{1},\,i_{2},\,\dots,\,i_{p})$ be an arbitrary sequence of 
elements of $I_{0}$ (with repetitions allowed), and set 
$\lambda:=\vpi_{i_{1}}+\vpi_{i_{2}}+ \cdots + \vpi_{i_{p}}$. 
Then the crystal $\BB(\lambda)_{\cl}$
is isomorphic, as a crystal with weights in $P_{\cl}$, 
to the tensor product 
$\BB_{\bi}:=
\BB(\vpi_{i_{1}})_{\cl} \otimes \BB(\vpi_{i_{2}})_{\cl} \otimes 
\cdots \otimes \BB(\vpi_{i_{p}})_{\cl}$ of KR crystals of one-column type. 
Moreover, in \cite[Theorem~4.1]{NS-Degree}, 
we proved that the degree function 
$\Deg=\Deg_{\lambda}:\BB(\lambda)_{\cl} \rightarrow \BZ_{\le 0}$ 
in \S\ref{subsec:degree} is identical, up to a constant, 
to the global energy function $D_{\bi}$ 
(which is called the energy function in \cite{LNSSS2}, and 
the right energy function in \cite{LS}; note that the order of 
tensor factors in tensor products of crystals in \cite{LS} 
is ``opposite'' to the one in this paper and \cite{LNSSS2}) on 
$\BB_{\bi}=
\BB(\vpi_{i_{1}})_{\cl} \otimes \BB(\vpi_{i_{2}})_{\cl} \otimes 
\cdots \otimes \BB(\vpi_{i_{p}})_{\cl}$ under the isomorphism 
$\Psi:\BB(\lambda)_{\cl} \stackrel{\sim}{\rightarrow} \BB_{\bi}$ 
of crystals above. 

Now we explain the relation between the degree function and 
the global energy function more precisely. 
Following \cite[\S3]{HKOTY} and \cite[\S3.3]{HKOTT} 
(see also \cite[\S4.1]{NS-Degree}), we define 
the global energy function 
$D_{\bi}:
 \BB_{\bi}=
 \BB(\vpi_{i_{1}})_{\cl} \otimes 
 \BB(\vpi_{i_{2}})_{\cl} \otimes \cdots \otimes 
 \BB(\vpi_{i_{p}})_{\cl} \rightarrow \BZ$ as follows. 
First we recall that there exists a unique isomorphism
\begin{align*}
&
\BB(\vpi_{i_{k}})_{\cl} \otimes 
\BB(\vpi_{i_{k+1}})_{\cl} \otimes \cdots \otimes 
\BB(\vpi_{i_{l-1}})_{\cl} \otimes \BB(\vpi_{i_{l}})_{\cl} \\
& \hspace*{20mm} \stackrel{\sim}{\rightarrow} 
\BB(\vpi_{i_{l}})_{\cl} \otimes 
\BB(\vpi_{i_{k}})_{\cl} \otimes \cdots \otimes 
\BB(\vpi_{i_{l-2}})_{\cl} \otimes \BB(\vpi_{i_{l-1}})_{\cl}
\end{align*}
of crystals, which is given as the composite of 
combinatorial $R$-matrices (see \cite[\S2.4]{NS-Degree}). 
For an element 
$\eta_{1} \otimes \eta_{2} \otimes \cdots \otimes 
 \eta_{p} \in \BB_{\bi}$, 
we define $\eta_{l}^{(k)} \in \BB(\vpi_{i_{l}})_{\cl}$, 
$1 \le k < l \le p$, to be the first factor
(which lies in $\BB(\vpi_{i_{l}})_{\cl}$) of 
the image of 
$\eta_{k} \otimes \eta_{k+1} \otimes 
\cdots \otimes \eta_{l} \in 
\BB(\vpi_{i_{k}})_{\cl} \otimes 
\BB(\vpi_{i_{k+1}})_{\cl} \otimes \cdots \otimes 
\BB(\vpi_{i_{l}})_{\cl}$ 
under the above isomorphism of crystals.
For convenience, we set $\eta_{l}^{(l)}:=\eta_{l}$ 
for $1 \le l \le p$. Furthermore, for each $1 \le k \le p$, 
take (and fix) an arbitrary element 
$\eta_{k}^{\flat} \in \BB(\vpi_{i_{k}})_{\cl}$ 
such that $f_{j}\eta_{k}^{\flat}=\bzero$ 
for all $j \in I_{0}$.  Then we set 
\begin{align*}
& 
D_{\bi}(
  \eta_{1} \otimes 
  \eta_{2} \otimes \cdots \otimes 
  \eta_{p}) = \nonumber \\
& \hspace*{15mm}
\sum_{1 \le k < l \le p} 
 H_{\vpi_{i_{k}}, \vpi_{i_{l}}}
 (\eta_{k} \otimes \eta_{l}^{(k+1)})
 + \sum_{k=1}^{p} 
 H_{\vpi_{i_{k}}, \vpi_{i_{k}}}
  (\eta^{\flat}_{k} \otimes \eta_{k}^{(1)}).
\end{align*}
Here, $H_{\vpi_{i_{k}}, \vpi_{i_{l}}}:
\BB(\vpi_{i_{k}})_{\cl} \otimes \BB(\vpi_{i_{l}})_{\cl}
\rightarrow \BZ$ is the local energy 
function, which is a unique $\BZ$-valued function on
$\BB(\vpi_{i_{k}})_{\cl} \otimes \BB(\vpi_{i_{l}})_{\cl}$ 
satisfying the conditions
\cite[(H1) and (H2) in Theorem~2.5.1]{NS-Degree}. 
Also, we define a constant $D_{\bi}^{\ext} \in \BZ$ by
\begin{equation*}
D_{\bi}^{\ext} := 
\sum_{k=1}^{p} 
 H_{\vpi_{i_{k}}, \vpi_{i_{k}}}
  (\eta^{\flat}_{k} \otimes \cl(\pi_{\vpi_{i_{k}}})). 
\end{equation*}
In \cite[Theorem~4.1]{NS-Degree}, we proved that 
for every $\eta \in \BB(\lambda)_{\cl}$, 
\begin{equation*}
\Deg(\eta)=D_{\bi}(\Psi(\eta))-D_{\bi}^{\ext}, 
\end{equation*}
where $\Psi:\BB(\lambda)_{\cl} \stackrel{\sim}{\rightarrow} \BB_{\bi}$ 
is the isomorphism of crystals above. 

\begin{rem}
We can verify that 
the function $D_{\bi} \circ \Psi:\BB(\lambda)_{\cl}
\rightarrow \BZ$ is a unique function on $\BB(\lambda)_{\cl}$ 
satisfying \cite[(3.2.1)]{NS-Degree} (with $\Deg$ replaced by 
$D_{\bi} \circ \Psi$) and the condition that $D_{\bi} \circ \Psi
(\cl(\pi_{\lambda})) = D_{\bi}^{\ext}$ 
(see \cite[Lemma~3.2.1\,(1)]{NS-Degree}). 
\end{rem}

%
\section{Quantum Lakshmibai-Seshadri paths.}
\label{sec:QLS}

%
\subsection{Parabolic quantum Bruhat graph.}
\label{subsec:def-QBG}

In this subsection, we fix a subset $J$ of $I_{0}$. Set
\begin{equation*}
W_{0,J}:=\langle r_{j} \mid j \in J \rangle \subset W_{0}.
\end{equation*}
It is well-known that each coset in $W_{0}/W_{0,J}$ has a unique element of 
minimal length, called the minimal coset representative for the coset; 
we denote by $W_{0}^{J} \subset W_{0}$ 
the set of minimal coset representatives for the cosets in 
$W_{0}/W_{0,J}$, and by $\mcr{\,\cdot\,}=\mcr{\,\cdot\,}_{J}:
W_{0} \twoheadrightarrow W_{0}^{J} \cong W_{0}/W_{0,J}$
the canonical projection. 
Also, we set $\Delta_{0,J}:=\Delta_{0} \cap 
\bigl(\bigoplus_{j \in J} \BZ \alpha_{j}\bigr)$, 
$\Delta_{0,J}^{\pm}:=
\Delta_{0}^{\pm} \cap \bigl(\bigoplus_{j \in J} \BZ \alpha_{j}\bigr)$, 
and 
$\rho:=(1/2)\sum_{\alpha \in \Delta_{0}^{+}}\alpha$, 
$\rho_{J}:=(1/2)\sum_{\alpha \in \Delta_{0,J}^{+}} \alpha$. 
%
%
\begin{dfn} \label{dfn:QBG}
The parabolic quantum Bruhat graph is 
a $(\Delta_{0}^{+} \setminus \Delta_{0,J}^{+})$-labeled, 
directed graph with vertex set $W_{0}^{J}$ and 
$(\Delta_{0}^{+} \setminus \Delta_{0,J}^{+})$-labeled, directed edges 
of the following form: $w \stackrel{\beta}{\rightarrow} \mcr{wr_{\beta}}$
for $w \in W_{0}^{J}$ and $\beta \in \Delta_{0}^{+} \setminus \Delta_{0,J}^{+}$
such that either
\begin{enumerate}
\renewcommand{\labelenumi}{(\roman{enumi})}

\item $\ell(\mcr{wr_{\beta}})=\ell(w)+1$, or 

\item $\ell(\mcr{wr_{\beta}})=\ell(w)-2\pair{\rho-\rho_{J}}{\beta^{\vee}}+1$;
\end{enumerate}
if (i) holds (resp., (ii) holds), then the edge is called a Bruhat edge 
(resp., a quantum edge). 
\end{dfn}

\begin{ex} \label{ex1}
Assume that $\Fg$ is of type $A_{2}^{(1)}$ (and hence $\Delta_{0}$ and $W_{0}$ 
are of type $A_{2}$), and $J=\emptyset$. 
Then the quantum Bruhat graph is as follows, where
$\theta=\alpha_{1}+\alpha_{2} \in \Delta_{0}^{+}$, the highest root of $A_{2}$: 
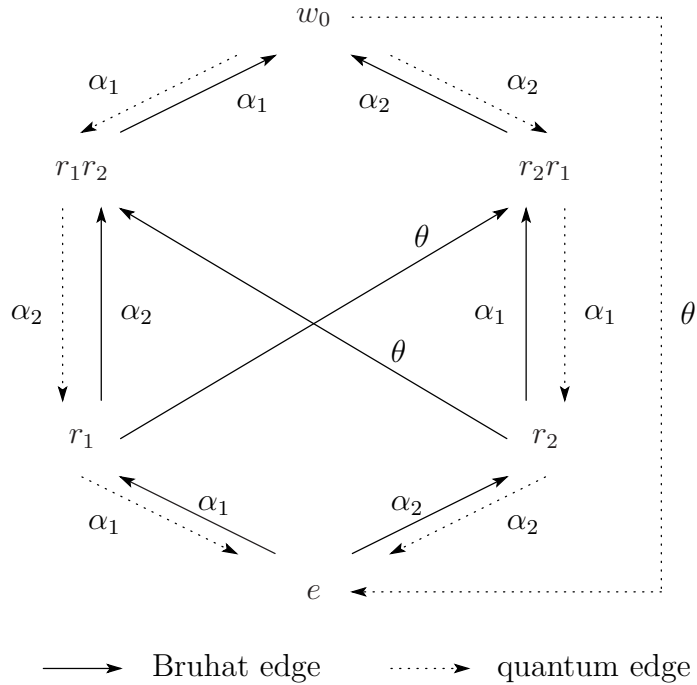
\begin{center}
\hspace{-5mm}
{\unitlength 0.1in%
\begin{picture}( 38.6000, 34.9200)(  2.4000,-38.2700)%
\put(22.0000,-4.0000){\makebox(0,0){$w_0$}}%
\put(10.0000,-12.0000){\makebox(0,0){$r_1r_2$}}%
\put(10.0000,-26.0000){\makebox(0,0){$r_1$}}%
\put(22.0000,-34.0000){\makebox(0,0){$e$}}%
\put(34.0000,-26.0000){\makebox(0,0){$r_2$}}%
\put(34.0000,-12.0000){\makebox(0,0){$r_2r_1$}}%
%
\special{pn 8}%
\special{pa 2000 3200}%
\special{pa 1200 2800}%
\special{fp}%
\special{sh 1}%
\special{pa 1200 2800}%
\special{pa 1251 2848}%
\special{pa 1248 2824}%
\special{pa 1269 2812}%
\special{pa 1200 2800}%
\special{fp}%
%
\special{pn 8}%
\special{pa 2400 3200}%
\special{pa 3200 2800}%
\special{fp}%
\special{sh 1}%
\special{pa 3200 2800}%
\special{pa 3131 2812}%
\special{pa 3152 2824}%
\special{pa 3149 2848}%
\special{pa 3200 2800}%
\special{fp}%
%
\special{pn 8}%
\special{pa 3300 2400}%
\special{pa 3300 1400}%
\special{fp}%
\special{sh 1}%
\special{pa 3300 1400}%
\special{pa 3280 1467}%
\special{pa 3300 1453}%
\special{pa 3320 1467}%
\special{pa 3300 1400}%
\special{fp}%
%
\special{pn 8}%
\special{pa 3200 1000}%
\special{pa 2400 600}%
\special{fp}%
\special{sh 1}%
\special{pa 2400 600}%
\special{pa 2451 648}%
\special{pa 2448 624}%
\special{pa 2469 612}%
\special{pa 2400 600}%
\special{fp}%
%
\special{pn 8}%
\special{pa 1200 1000}%
\special{pa 2000 600}%
\special{fp}%
\special{sh 1}%
\special{pa 2000 600}%
\special{pa 1931 612}%
\special{pa 1952 624}%
\special{pa 1949 648}%
\special{pa 2000 600}%
\special{fp}%
%
\special{pn 8}%
\special{pa 1100 2400}%
\special{pa 1100 1400}%
\special{fp}%
\special{sh 1}%
\special{pa 1100 1400}%
\special{pa 1080 1467}%
\special{pa 1100 1453}%
\special{pa 1120 1467}%
\special{pa 1100 1400}%
\special{fp}%
%
\special{pn 8}%
\special{pa 1000 2800}%
\special{pa 1800 3200}%
\special{dt 0.045}%
\special{sh 1}%
\special{pa 1800 3200}%
\special{pa 1749 3152}%
\special{pa 1752 3176}%
\special{pa 1731 3188}%
\special{pa 1800 3200}%
\special{fp}%
%
\special{pn 8}%
\special{pa 3400 2800}%
\special{pa 2600 3200}%
\special{dt 0.045}%
\special{sh 1}%
\special{pa 2600 3200}%
\special{pa 2669 3188}%
\special{pa 2648 3176}%
\special{pa 2651 3152}%
\special{pa 2600 3200}%
\special{fp}%
%
\special{pn 8}%
\special{pa 3500 1400}%
\special{pa 3500 2400}%
\special{dt 0.045}%
\special{sh 1}%
\special{pa 3500 2400}%
\special{pa 3520 2333}%
\special{pa 3500 2347}%
\special{pa 3480 2333}%
\special{pa 3500 2400}%
\special{fp}%
%
\special{pn 8}%
\special{pa 900 1400}%
\special{pa 900 2400}%
\special{dt 0.045}%
\special{sh 1}%
\special{pa 900 2400}%
\special{pa 920 2333}%
\special{pa 900 2347}%
\special{pa 880 2333}%
\special{pa 900 2400}%
\special{fp}%
%
\special{pn 8}%
\special{pa 1800 600}%
\special{pa 1000 1000}%
\special{dt 0.045}%
\special{sh 1}%
\special{pa 1000 1000}%
\special{pa 1069 988}%
\special{pa 1048 976}%
\special{pa 1051 952}%
\special{pa 1000 1000}%
\special{fp}%
%
\special{pn 8}%
\special{pa 2600 600}%
\special{pa 3400 1000}%
\special{dt 0.045}%
\special{sh 1}%
\special{pa 3400 1000}%
\special{pa 3349 952}%
\special{pa 3352 976}%
\special{pa 3331 988}%
\special{pa 3400 1000}%
\special{fp}%
%
\special{pn 8}%
\special{pa 1200 2600}%
\special{pa 3200 1400}%
\special{fp}%
\special{sh 1}%
\special{pa 3200 1400}%
\special{pa 3133 1417}%
\special{pa 3154 1427}%
\special{pa 3153 1451}%
\special{pa 3200 1400}%
\special{fp}%
%
\special{pn 8}%
\special{pa 3200 2600}%
\special{pa 1200 1400}%
\special{fp}%
\special{sh 1}%
\special{pa 1200 1400}%
\special{pa 1247 1451}%
\special{pa 1246 1427}%
\special{pa 1267 1417}%
\special{pa 1200 1400}%
\special{fp}%
%
\special{pn 8}%
\special{pa 2400 400}%
\special{pa 4000 400}%
\special{dt 0.045}%
%
\special{pn 8}%
\special{pa 4000 400}%
\special{pa 4000 3400}%
\special{dt 0.045}%
%
\special{pn 8}%
\special{pa 4000 3400}%
\special{pa 2400 3400}%
\special{dt 0.045}%
\special{sh 1}%
\special{pa 2400 3400}%
\special{pa 2467 3420}%
\special{pa 2453 3400}%
\special{pa 2467 3380}%
\special{pa 2400 3400}%
\special{fp}%
\put(26.0000,-30.0000){\makebox(0,0)[lb]{$\alpha_2$}}%
\put(16.0000,-30.0000){\makebox(0,0)[lb]{$\alpha_1$}}%
\put(26.0000,-22.0000){\makebox(0,0)[lb]{$\theta$}}%
\put(28.0000,-16.0000){\makebox(0,0)[rb]{$\theta$}}%
\put(36.0000,-20.0000){\makebox(0,0)[lb]{$\alpha_1$}}%
\put(8.0000,-20.0000){\makebox(0,0)[rb]{$\alpha_2$}}%
\put(32.0000,-8.0000){\makebox(0,0)[lb]{$\alpha_2$}}%
\put(12.0000,-8.0000){\makebox(0,0)[rb]{$\alpha_1$}}%
\put(41.0000,-20.0000){\makebox(0,0)[lb]{$\theta$}}%
%
\special{pn 8}%
\special{pa 800 3800}%
\special{pa 1200 3800}%
\special{fp}%
\special{sh 1}%
\special{pa 1200 3800}%
\special{pa 1133 3780}%
\special{pa 1147 3800}%
\special{pa 1133 3820}%
\special{pa 1200 3800}%
\special{fp}%
\put(18.0000,-38.0000){\makebox(0,0){Bruhat edge}}%
%
\special{pn 8}%
\special{pa 2600 3800}%
\special{pa 3000 3800}%
\special{dt 0.045}%
\special{sh 1}%
\special{pa 3000 3800}%
\special{pa 2933 3780}%
\special{pa 2947 3800}%
\special{pa 2933 3820}%
\special{pa 3000 3800}%
\special{fp}%
\put(36.5000,-38.0000){\makebox(0,0){quantum edge}}%
\put(12.0000,-30.0000){\makebox(0,0)[rt]{$\alpha_1$}}%
\put(32.0000,-30.0000){\makebox(0,0)[lt]{$\alpha_2$}}%
\put(32.0000,-20.0000){\makebox(0,0)[rb]{$\alpha_1$}}%
\put(12.0000,-20.0000){\makebox(0,0)[lb]{$\alpha_2$}}%
\put(18.0000,-8.0000){\makebox(0,0)[lt]{$\alpha_1$}}%
\put(26.0000,-8.0000){\makebox(0,0)[rt]{$\alpha_2$}}%
\end{picture}}%

\end{center}
\end{ex}

Let $x,\,y \in W_{0}^{J}$. A directed path $\bd$
from $y$ to $x$ in the parabolic quantum Bruhat graph is, 
by definition, a pair of a sequence $w_{0},\,w_{1},\,\dots,\,w_{n}$ 
of elements in $W_{0}^{J}$ and a sequence 
$\beta_{1},\,\beta_{2},\,\dots,\,\beta_{n}$ of 
elements in $\Delta_{0}^{+} \setminus \Delta_{0,J}^{+}$ 
such that in the parabolic quantum Bruhat graph, 
%
%
\begin{equation} \label{eq:dp}
\bd : 
x=
 w_{0} \stackrel{\beta_{1}}{\leftarrow}
 w_{1} \stackrel{\beta_{2}}{\leftarrow} \cdots 
       \stackrel{\beta_{n}}{\leftarrow}
 w_{n}=y. 
\end{equation}
A directed path $\bd$ from $y$ to $x$ is said to be 
shortest if its length $n$ is minimal among 
all possible directed paths from $y$ to $x$; 
let $\len{x}{y}$ denote the length of a shortest directed path 
from $y$ to $x$ in the parabolic quantum Bruhat graph.
Also, we define the weight $\wt(\bd) \in 
Q^{\vee}=\bigoplus_{j \in I_{0}}\BZ \alpha_{j}^{\vee}$ of 
a directed path of the form \eqref{eq:dp} by
%
%
\begin{equation} \label{eq:wtdp}
\wt (\bd) := \sum_{
 \begin{subarray}{c}
 1 \le k \le n\,; \\[1mm]
 \text{$w_{k-1} \stackrel{\beta_{k}}{\leftarrow}
 w_{k}$ is } \\ \text{a quantum edge}
 \end{subarray}}
\beta_{k}^{\vee}.
\end{equation}

We recall the following proposition from 
\cite[Theorem~6.5]{LNSSS1}. 
%
%
\begin{prop} \label{prop:QB}
Set $\Lambda:=\cl(\lambda) \in P_{\cl}$. 
\begin{enu}

\item 
Let $w \in W_{0}^{J}$ and $\beta \in \Delta_{0}^{+} \setminus \Delta_{0,J}^{+}$ 
be such that $\mcr{wr_{\beta}} \stackrel{\beta}{\longleftarrow} w$ 
in the parabolic quantum Bruhat graph. We set
\begin{equation*}
\xi:=
\begin{cases}
 w\beta & 
 \text{\rm if $\mcr{wr_{\beta}} \stackrel{\beta}{\longleftarrow} w$ 
 is a Bruhat edge}, \\[3mm]
 w\beta+\delta & 
 \text{\rm if $\mcr{wr_{\beta}} \stackrel{\beta}{\longleftarrow} w$ 
 is a quantum edge}.
\end{cases}
\end{equation*}
Then, $\xi \in \prr$, and $r_{\xi}\nu \gtrdot \nu$ 
for all $\nu \in W\lambda$ such that $\cl(\nu)=w\Lambda$. 

\item
Let $\mu,\,\nu \in W\lambda$ be such that $\mu \gtrdot \nu$, 
and let $\xi \in \prr$ be the positive real root such that 
$r_{\xi}\mu=\nu$; recall from Remark~\ref{rem:cover} that 
$\xi \in \Delta_{0}^{+} \sqcup \bigl\{-\gamma+\delta \mid 
\gamma \in \Delta_{0}^{+}\bigr\}$. Let $w \in W_{0}^{J}$ be 
a unique element in $W_{0}^{J}$ such that 
$\cl(\nu)=w\Lambda$, and set
\begin{equation*}
\beta:=
\begin{cases}
 w^{-1}\xi & \text{\rm if $\xi \in \Delta_{0}^{+}$}, \\[1.5mm]
 w^{-1}(\xi-\delta) & \text{\rm if $\xi \in \bigl\{-\gamma+\delta \mid 
\gamma \in \Delta_{0}^{+}\bigr\}$}. 
\end{cases}
\end{equation*}
Then, $\beta \in \Delta_{0}^{+} \setminus \Delta_{J}^{+}$, 
and $\mcr{wr_{\beta}} \stackrel{\beta}{\longleftarrow} w$ 
in the parabolic quantum Bruhat graph; note that 
$\cl(\mu)=\mcr{wr_{\beta}}\Lambda$. Moreover, 
the edge $\mcr{wr_{\beta}} \stackrel{\beta}{\longleftarrow} w$ is 
a Bruhat {\rm(}resp., quantum{\rm)} edge if 
$\xi \in \Delta_{0}^{+}$ {\rm(}resp., 
$\xi \in \bigl\{-\gamma+\delta \mid 
\gamma \in \Delta_{0}^{+}\bigr\}${\rm)}. 
\end{enu}
\end{prop}

%
\subsection{Definition of quantum Lakshmibai-Seshadri paths.}
\label{subsec:QLS}

In this subsection, 
we fix a level-zero dominant integral weight
$\lambda \in \sum_{i \in I_{0}} \BZ_{\ge 0} \vpi_{i}$, and 
set $\Lambda:=\cl(\lambda)$ for simplicity of notation. 
Also, we set 
\begin{equation*}
J:=\bigl\{j \in I_{0} \mid \pair{\Lambda}{\alpha_{j}^{\vee}}=0 \bigr\} 
 \subset I_{0}.
\end{equation*}
%
%
\begin{dfn} \label{dfn:QBG-achain}
Let $x,\,y \in W_{0}^{J}$, and let $\sigma \in \BQ$ 
be such that $0 < \sigma < 1$. A directed $\sigma$-path 
from $y$ to $x$ is, by definition, a directed path 
\begin{equation*}
x=w_{0} \stackrel{\beta_{1}}{\leftarrow} w_{1}
\stackrel{\beta_{2}}{\leftarrow} w_{2} 
\stackrel{\beta_{3}}{\leftarrow} \cdots 
\stackrel{\beta_{n}}{\leftarrow} w_{n}=y
\end{equation*}
from $y$ to $x$ in the parabolic quantum Bruhat graph 
satisfying the condition that 
\begin{equation*}
\sigma \pair{\Lambda}{\beta_{k}^{\vee}} \in \BZ 
\quad \text{for all $1 \le k \le n$}.
\end{equation*}
\end{dfn}
%
%
\begin{rem} \label{rem:QB}
Keep the notation and setting of Proposition~\ref{prop:QB}\,(1). 
Let $0 < \sigma < 1$ be a rational number. 
If an edge $\mcr{wr_{\beta}} \stackrel{\beta}{\longleftarrow} w$ satisfies 
$\sigma \pair{\Lambda}{\beta^{\vee}} \in \BZ$, 
then $r_{\xi}\nu \gtrdot \nu$ is a $\sigma$-chain 
for $(r_{\xi}\nu,\,\nu)$. Indeed, we have
$\sigma\pair{\nu}{\xi^{\vee}} = 
\sigma \pair{w\Lambda}{w\beta^{\vee}}= 
\sigma \pair{\Lambda}{\beta^{\vee}} \in \BZ$.
\end{rem}
\begin{ex} \label{ex2}
Assume that $\Fg$ is of type $A_{2}^{(1)}$, and $\lambda=2\vpi_{1}+\vpi_{2}$. 
Then, $J$ is the empty set, and hence the corresponding 
(parabolic) quantum Bruhat graph is the one in Example~\ref{ex1}. 
In the figure below, the symbol $[a]$ on an edge indicates that the value 
of $\Lambda=\cl(\lambda)$ at the coroot of the label of the edge is equal to $a$: 
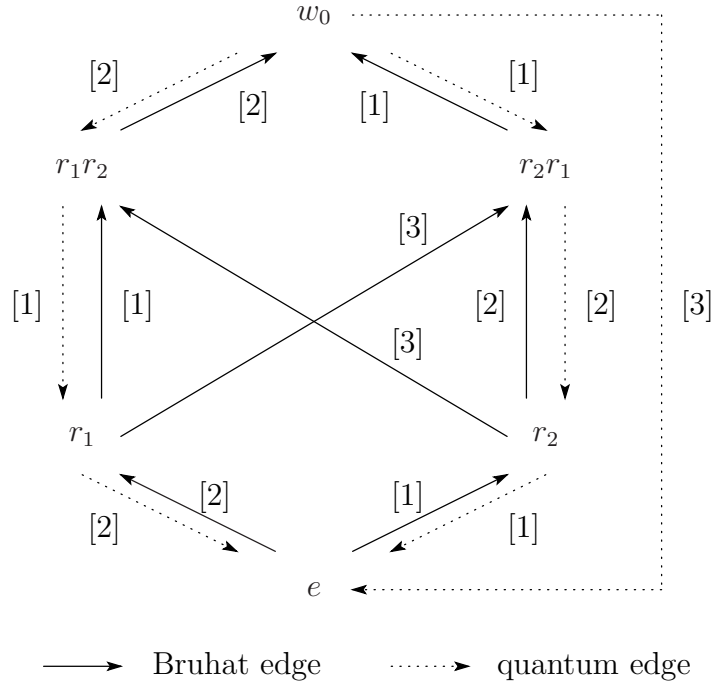
\begin{center}
{\unitlength 0.1in%
\begin{picture}( 34.4000, 34.9200)(  6.6000,-38.2700)%
\put(22.0000,-4.0000){\makebox(0,0){$w_0$}}%
\put(10.0000,-12.0000){\makebox(0,0){$r_1r_2$}}%
\put(10.0000,-26.0000){\makebox(0,0){$r_1$}}%
\put(22.0000,-34.0000){\makebox(0,0){$e$}}%
\put(34.0000,-26.0000){\makebox(0,0){$r_2$}}%
\put(34.0000,-12.0000){\makebox(0,0){$r_2r_1$}}%
%
\special{pn 8}%
\special{pa 2000 3200}%
\special{pa 1200 2800}%
\special{fp}%
\special{sh 1}%
\special{pa 1200 2800}%
\special{pa 1251 2848}%
\special{pa 1248 2824}%
\special{pa 1269 2812}%
\special{pa 1200 2800}%
\special{fp}%
%
\special{pn 8}%
\special{pa 2400 3200}%
\special{pa 3200 2800}%
\special{fp}%
\special{sh 1}%
\special{pa 3200 2800}%
\special{pa 3131 2812}%
\special{pa 3152 2824}%
\special{pa 3149 2848}%
\special{pa 3200 2800}%
\special{fp}%
%
\special{pn 8}%
\special{pa 3300 2400}%
\special{pa 3300 1400}%
\special{fp}%
\special{sh 1}%
\special{pa 3300 1400}%
\special{pa 3280 1467}%
\special{pa 3300 1453}%
\special{pa 3320 1467}%
\special{pa 3300 1400}%
\special{fp}%
%
\special{pn 8}%
\special{pa 3200 1000}%
\special{pa 2400 600}%
\special{fp}%
\special{sh 1}%
\special{pa 2400 600}%
\special{pa 2451 648}%
\special{pa 2448 624}%
\special{pa 2469 612}%
\special{pa 2400 600}%
\special{fp}%
%
\special{pn 8}%
\special{pa 1200 1000}%
\special{pa 2000 600}%
\special{fp}%
\special{sh 1}%
\special{pa 2000 600}%
\special{pa 1931 612}%
\special{pa 1952 624}%
\special{pa 1949 648}%
\special{pa 2000 600}%
\special{fp}%
%
\special{pn 8}%
\special{pa 1100 2400}%
\special{pa 1100 1400}%
\special{fp}%
\special{sh 1}%
\special{pa 1100 1400}%
\special{pa 1080 1467}%
\special{pa 1100 1453}%
\special{pa 1120 1467}%
\special{pa 1100 1400}%
\special{fp}%
%
\special{pn 8}%
\special{pa 1000 2800}%
\special{pa 1800 3200}%
\special{dt 0.045}%
\special{sh 1}%
\special{pa 1800 3200}%
\special{pa 1749 3152}%
\special{pa 1752 3176}%
\special{pa 1731 3188}%
\special{pa 1800 3200}%
\special{fp}%
%
\special{pn 8}%
\special{pa 3400 2800}%
\special{pa 2600 3200}%
\special{dt 0.045}%
\special{sh 1}%
\special{pa 2600 3200}%
\special{pa 2669 3188}%
\special{pa 2648 3176}%
\special{pa 2651 3152}%
\special{pa 2600 3200}%
\special{fp}%
%
\special{pn 8}%
\special{pa 3500 1400}%
\special{pa 3500 2400}%
\special{dt 0.045}%
\special{sh 1}%
\special{pa 3500 2400}%
\special{pa 3520 2333}%
\special{pa 3500 2347}%
\special{pa 3480 2333}%
\special{pa 3500 2400}%
\special{fp}%
%
\special{pn 8}%
\special{pa 900 1400}%
\special{pa 900 2400}%
\special{dt 0.045}%
\special{sh 1}%
\special{pa 900 2400}%
\special{pa 920 2333}%
\special{pa 900 2347}%
\special{pa 880 2333}%
\special{pa 900 2400}%
\special{fp}%
%
\special{pn 8}%
\special{pa 1800 600}%
\special{pa 1000 1000}%
\special{dt 0.045}%
\special{sh 1}%
\special{pa 1000 1000}%
\special{pa 1069 988}%
\special{pa 1048 976}%
\special{pa 1051 952}%
\special{pa 1000 1000}%
\special{fp}%
%
\special{pn 8}%
\special{pa 2600 600}%
\special{pa 3400 1000}%
\special{dt 0.045}%
\special{sh 1}%
\special{pa 3400 1000}%
\special{pa 3349 952}%
\special{pa 3352 976}%
\special{pa 3331 988}%
\special{pa 3400 1000}%
\special{fp}%
%
\special{pn 8}%
\special{pa 1200 2600}%
\special{pa 3200 1400}%
\special{fp}%
\special{sh 1}%
\special{pa 3200 1400}%
\special{pa 3133 1417}%
\special{pa 3154 1427}%
\special{pa 3153 1451}%
\special{pa 3200 1400}%
\special{fp}%
%
\special{pn 8}%
\special{pa 3200 2600}%
\special{pa 1200 1400}%
\special{fp}%
\special{sh 1}%
\special{pa 1200 1400}%
\special{pa 1247 1451}%
\special{pa 1246 1427}%
\special{pa 1267 1417}%
\special{pa 1200 1400}%
\special{fp}%
%
\special{pn 8}%
\special{pa 2400 400}%
\special{pa 4000 400}%
\special{dt 0.045}%
%
\special{pn 8}%
\special{pa 4000 400}%
\special{pa 4000 3400}%
\special{dt 0.045}%
%
\special{pn 8}%
\special{pa 4000 3400}%
\special{pa 2400 3400}%
\special{dt 0.045}%
\special{sh 1}%
\special{pa 2400 3400}%
\special{pa 2467 3420}%
\special{pa 2453 3400}%
\special{pa 2467 3380}%
\special{pa 2400 3400}%
\special{fp}%
\put(26.0000,-30.0000){\makebox(0,0)[lb]{[1]}}%
\put(16.0000,-30.0000){\makebox(0,0)[lb]{[2]}}%
\put(26.0000,-22.0000){\makebox(0,0)[lb]{[3]}}%
\put(28.0000,-16.0000){\makebox(0,0)[rb]{[3]}}%
\put(36.0000,-20.0000){\makebox(0,0)[lb]{[2]}}%
\put(8.0000,-20.0000){\makebox(0,0)[rb]{[1]}}%
\put(32.0000,-8.0000){\makebox(0,0)[lb]{[1]}}%
\put(12.0000,-8.0000){\makebox(0,0)[rb]{[2]}}%
\put(41.0000,-20.0000){\makebox(0,0)[lb]{[3]}}%
%
\special{pn 8}%
\special{pa 800 3800}%
\special{pa 1200 3800}%
\special{fp}%
\special{sh 1}%
\special{pa 1200 3800}%
\special{pa 1133 3780}%
\special{pa 1147 3800}%
\special{pa 1133 3820}%
\special{pa 1200 3800}%
\special{fp}%
\put(18.0000,-38.0000){\makebox(0,0){Bruhat edge}}%
%
\special{pn 8}%
\special{pa 2600 3800}%
\special{pa 3000 3800}%
\special{dt 0.045}%
\special{sh 1}%
\special{pa 3000 3800}%
\special{pa 2933 3780}%
\special{pa 2947 3800}%
\special{pa 2933 3820}%
\special{pa 3000 3800}%
\special{fp}%
\put(36.5000,-38.0000){\makebox(0,0){quantum edge}}%
\put(12.0000,-30.0000){\makebox(0,0)[rt]{[2]}}%
\put(32.0000,-30.0000){\makebox(0,0)[lt]{[1]}}%
\put(32.0000,-20.0000){\makebox(0,0)[rb]{[2]}}%
\put(12.0000,-20.0000){\makebox(0,0)[lb]{[1]}}%
\put(18.0000,-8.0000){\makebox(0,0)[lt]{[2]}}%
\put(26.0000,-8.0000){\makebox(0,0)[rt]{[1]}}%
\end{picture}}%

\end{center}
From this, we see that the directed edges 
$r_{1} \stackrel{\theta}{\longrightarrow} r_{2}r_{1}$, 
$w_{0} \stackrel{\theta}{\longrightarrow} e$, and 
$r_{2} \stackrel{\theta}{\longrightarrow} r_{1}r_{2}$ are $(1/3)$-paths, and hence $(2/3)$-paths.  
Also, we see that the directed edges 
$e \stackrel{\alpha_1}{\longrightarrow} r_{1}$, 
$r_{1}r_{2} \stackrel{\alpha_1}{\longrightarrow} w_{0}$, and 
$r_{2}r_{1} \stackrel{\alpha_1}{\longrightarrow} r_{2}$ are $(1/2)$-paths. 
\end{ex}
%
%
\begin{dfn} \label{dfn:qLS}
Let us denote by $\ti{\BB}(\lambda)_{\cl}$ (resp., $\ha{\BB}(\lambda)_{\cl}$) 
the set of all pairs $\eta=(\ud{x}\,;\,\ud{\sigma})$ of 
a sequence $\ud{x}\,:\,x_{1},\,x_{2},\,\dots,\,x_{s}$ of 
elements in $W_{0}^{J}$, with $x_{k} \ne x_{k+1}$ 
for $1 \le k \le s-1$, and a sequence 
$\ud{\sigma}\,:\,
 0=\sigma_{0} < \sigma_{1} < \cdots < \sigma_{s}=1$ of rational numbers
 satisfying the condition that 
there exists a directed $\sigma_{k}$-path 
(resp., directed $\sigma_{k}$-path of length $\len{x_{k}}{x_{k+1}}$) 
from $x_{k+1}$ to $x_{k}$ for each $1 \le k \le s-1$. We call an element of 
$\ti{\BB}(\lambda)_{\cl}$ a quantum Lakshmibai-Seshadri path 
of shape $\lambda$. 
\end{dfn}

\begin{ex} \label{ex3}
Keep the notation and setting of Example~\ref{ex2}. We can check that
\begin{align*}
\eta_1 & = (r_{2},\,r_{2}r_{1},\,r_{1}\,;\,0,\,1/2,\,2/3,\,1), \\
\eta_2 & = (r_{1},\,e,\,w_{0}\,;\,0,\,1/2,\,2/3,\,1), \\
\eta_3 & = (e,\,w_{0},\,r_{1}r_{2}\,;\,0,\,1/3,\,1/2,\,1)
\end{align*}
are quantum LS paths of shape $\lambda$. 
\end{ex}

Let $\eta=(x_{1},\,x_{2},\,\dots,\,x_{s}\,;\,
\sigma_{0},\,\sigma_{1},\,\dots,\,\sigma_{s})$ 
be a rational path, that is, 
a pair of a sequence $x_{1},\,x_{2},\,\dots,\,x_{s}$ 
of elements in $W_{0}^{J}$, with $x_{k} \ne x_{k+1}$ 
for $1 \le k \le s-1$, and a sequence 
$0=\sigma_{0} < \sigma_{1} < \cdots < \sigma_{s}=1$ of 
rational numbers. 
We identify $\eta$ with the following piecewise-linear, 
continuous map 
$\eta:[0,1] \rightarrow \BR \otimes_{\BZ} P_{\cl}$ 
(cf. \eqref{eq:path}): 
%
%
\begin{equation} \label{eq:QBG-path}
\eta(t)=\sum_{l=1}^{k-1}
(\sigma_{l}-\sigma_{l-1})x_{l}\Lambda+
(t-\sigma_{k-1})x_{k}\Lambda \quad 
\text{for $\sigma_{k-1} \le t \le \sigma_{k}$, $1 \le k \le s$};
\end{equation}
here we note that the map $W_{0}^{J} \rightarrow W_{0}\Lambda$, 
$w \mapsto w\Lambda$, is bijective.

We know the following from 
\cite[Theorem 4.1.1]{LNSSS3} (see also \cite{LNSSS2}). 
%
%
\begin{thm} \label{thm:QLS}
With the notation and setting above, we have
\begin{equation*}
\ha{\BB}(\lambda)_{\cl}=
\ti{\BB}(\lambda)_{\cl}=\BB(\lambda)_{\cl}.
\end{equation*}
\end{thm}
%
%
\section{Main result.}
\label{sec:main}

%
\subsection{Description of the degree function in terms of 
the parabolic quantum Bruhat graph.}
\label{subsec:main}

As in \S\ref{subsec:QLS}, 
we fix a level-zero dominant integral weight
$\lambda \in \sum_{j \in I_{0}} \BZ_{\ge 0} \vpi_{j}$, 
and set $J=\bigl\{j \in I_{0} \mid \pair{\Lambda}{\alpha^{\vee}_{j}}=0 \bigr\}$, 
where $\Lambda:=\cl(\lambda)$. 

Let $\eta \in \BB(\lambda)_{\cl}$. 
By Theorem~\ref{thm:QLS}, we can write 
$\eta$ in the form: 
\begin{equation*}
\eta = (x_{1},\,x_{2},\,\dots,\,x_{s}\,;\,
\sigma_{0},\,\sigma_{1},\,\dots,\,\sigma_{s}) 
\in \ha{\BB}(\lambda)_{\cl}.
\end{equation*}
For each $1 \le p \le s-1$, let $\bd_{p}$ denote a directed 
$\sigma_{p}$-path from $x_{p+1}$ to $x_{p}$ 
of length $\len{x_{p}}{x_{p+1}}$; observe that 
the value $\pair{\Lambda}{\wt(\bd_{p})}$ 
does not depend on the choice of 
such a directed $\sigma_{p}$-path $\bd_{p}$. 
Indeed, if $\bd_{p}'$ is another 
directed $\sigma_{p}$-path from $x_{p+1}$ to $x_{p}$ 
of length $\len{x_{p}}{x_{p+1}}$, 
then it follows from \cite[Proposition~8.1]{LNSSS1} that 
$\wt(\bd_{p})-\wt(\bd_{p}') \in Q_{J}^{\vee}:=
\bigoplus_{j \in J}\BZ \alpha_{j}^{\vee}$. Since 
$J=\bigl\{j \in I_{0} \mid 
\pair{\Lambda}{\alpha^{\vee}_{j}}=0 \bigr\}$ by the definition, 
we have
\begin{equation*}
\pair{\Lambda}{\wt(\bd_{p})-\wt(\bd_{p}')}=0, \quad
\text{and hence} \quad \pair{\Lambda}{\wt(\bd_{p})}=
\pair{\Lambda}{\wt(\bd_{p}')}.
\end{equation*}
Now, we define
%
%
\begin{equation} \label{eq:tinu}
\ti{\nu}_{1}:=x_{1}\lambda, \qquad 
\ti{\nu}_{p}:=x_{p}\lambda+
 \left(\sum_{u=1}^{p-1} \Bpair{\Lambda}{\wt(\bd_{u})}\right)\delta
\quad \text{for $2 \le p \le s$},
\end{equation}
and set 
\begin{equation*}
\ti{\pi}_{\eta}:=
  (\ti{\nu}_{1},\,\ti{\nu}_{2},\,\dots,\,\ti{\nu}_{s}\,;\,
   \sigma_{0},\,\sigma_{1},\,\dots,\,\sigma_{s}). 
\end{equation*}
%
%
The following is the main result of this paper; 
its proof will be given in the next subsection. 
%
%
\begin{thm} \label{thm:main}
Keep the notation above. Then, the element 
$\ti{\pi}_{\eta}$ defined above 
is identical to the element $\pi_{\eta} \in \BB_{0}(\lambda) \subset \BB(\lambda)$ 
in Proposition~\ref{prop:degree}. Moreover, we have 
%
%
\begin{equation} \label{eq:deg}
\Deg(\eta)=-\sum_{p=1}^{s-1} 
 (1-\sigma_{p})\Bpair{\Lambda}{\wt(\bd_{p})}. 
\end{equation}
\end{thm}

\begin{rem}
The formula \eqref{eq:deg} is identical to the one obtained 
in \cite[Theorem~4.5]{LNSSS2}, but the proof given there is 
completely different from the proof given in the next subsection.
\end{rem}

\begin{ex} \label{ex4}
Keep the notation and setting of Examples~\ref{ex2} and \ref{ex3}. 
Let us compute $\Deg(\eta_{1})$. It is obvious that 
$r_{2} \stackrel{\alpha_1}{\longleftarrow} r_{2}r_{1}$ 
(resp., $r_{2}r_{1} \stackrel{\theta}{\longleftarrow} r_{1}$) 
is a shortest directed path from $r_{2}r_{1}$ to $r_{2}$ 
(resp., from $r_{1}$ to $r_{2}r_{1}$). Because 
$r_{2} \stackrel{\alpha_1}{\longleftarrow} r_{2}r_{1}$ 
(resp., $r_{2}r_{1} \stackrel{\theta}{\longleftarrow} r_{1}$) is a quantum edge 
(resp., Bruhat edge), it follows from the definition \eqref{eq:wtdp} of 
the weight of a directed path that 
\begin{equation*}
\wt (r_{2} \stackrel{\alpha_1}{\longleftarrow} r_{2}r_{1}) = \alpha_{1}^{\vee}
 \quad \text{and} \quad
\wt (r_{2}r_{1} \stackrel{\theta}{\longleftarrow} r_{1}) = 0.
\end{equation*}
Hence, by Theorem~\ref{thm:main}, we have
\begin{align*}
\Deg(\eta_{1}) & = 
 -\left(1-\frac{1}{2}\right) \Bpair{\Lambda}{ \wt (r_{2} \stackrel{\alpha_1}{\longleftarrow} r_{2}r_{1}) }
 -\left(1-\frac{2}{3}\right) \Bpair{\Lambda}{ \wt (r_{2}r_{1} \stackrel{\theta}{\longleftarrow} r_{1}) } \\[1.5mm]
& = 
 -\left(1-\frac{1}{2}\right) \underbrace{\Bpair{\Lambda}{\alpha_{1}^{\vee}}}_{=2}
 -\left(1-\frac{2}{3}\right) \pair{\Lambda}{0} = -1. 
\end{align*}
Similarly, we have
\begin{equation*}
\Deg(\eta_{2}) = 
 -\left(1-\frac{1}{2}\right) 
 \Bpair{\Lambda}{ \underbrace{\wt(r_{1} \stackrel{\alpha_1}{\longleftarrow} e)}_{=0} } 
 -\left(1-\frac{2}{3}\right) 
 \Bpair{\Lambda}{ \underbrace{\wt(e \stackrel{\theta}{\longleftarrow} w_0)}_{=\theta^{\vee}} } = -1, 
\end{equation*}
\begin{equation*}
\Deg(\eta_{3}) = 
 -\left(1-\frac{1}{3}\right) 
 \Bpair{\Lambda}{ \underbrace{\wt(e \stackrel{\theta}{\longleftarrow} w_0)}_{=\theta^{\vee}} } 
 -\left(1-\frac{1}{2}\right) 
 \Bpair{\Lambda}{ \underbrace{\wt(w_0 \stackrel{\alpha_1}{\longleftarrow} r_1r_2)}_{=0} } = -2.
\end{equation*}
\end{ex}

%
\subsection{Proof of Theorem~\ref{thm:main}.}
\label{subsec:prf-main}

Keep the notation of the previous subsection. 
First we claim that $\ti{\pi}_{\eta} \in \BB(\lambda)$. 
We will show by induction on $p$ that 
$\ti{\nu}_{p} \in W\lambda$ for all $1 \le p \le s$. 
If $p=1$, then the assertion is obvious 
from the definition: $\ti{\nu}_{1}=x_{1}\lambda$. 
Assume now that $s-1 \ge p \ge 1$, and 
$\bd_{p}$ is of the form
\begin{equation*}
\bd_{p} : 
x_{p}=w_{0} \stackrel{\beta_{1}}{\longleftarrow} 
w_{1} \stackrel{\beta_{2}}{\longleftarrow} 
w_{2} \stackrel{\beta_{3}}{\longleftarrow} \cdots 
\stackrel{\beta_{n}}{\longleftarrow} w_{n}=x_{p+1}. 
\end{equation*}
For each $1 \le k \le n$, we define $\xi_{k} \in \prr$ as follows 
(see Proposition~\ref{prop:QB}): 
%
%
\begin{equation} \label{eq:xi}
\xi_{k}=
\begin{cases}
w_{k}\beta_{k} 
 & \text{if $w_{k-1}=\mcr{w_{k}r_{\beta_{k}}} 
   \stackrel{\beta_{k}}{\longleftarrow} w_{k}$ is a Bruhat edge}, \\[1.5mm]
w_{k}\beta_{k}+\delta
 & \text{if $w_{k-1}=\mcr{w_{k}r_{\beta_{k}}} 
   \stackrel{\beta_{k}}{\longleftarrow} w_{k}$ is a quantum edge}. 
\end{cases}
\end{equation}
Then, for $0 \le k \le n$, we obtain
%
%
\begin{equation} \label{eq:timu}
\ti{\mu}_{k}:=
r_{\xi_{k}} \cdots r_{\xi_{2}}r_{\xi_{1}}\ti{\nu}_{p} =
w_{k}\lambda+
 \left(\sum_{u=1}^{p-1} \Bpair{\Lambda}{\wt(\bd_{u})}\right)\delta+
 \left(\sum_{l \in [1,\,k]_{\Q}} \Bpair{\Lambda}{\beta_{l}^{\vee}}\right)\delta,
\end{equation}
where $[1,\,k]_{\Q} := \bigl\{ 1 \le l \le k \mid 
\text{$w_{l-1}=\mcr{w_{l}r_{\beta_{l}}} 
   \stackrel{\beta_{l}}{\longleftarrow} w_{l}$ is a quantum edge} \bigr\}$. 
Indeed, this equation follows by induction on $k$. 
If $k=0$, then equation \eqref{eq:timu} is obvious by \eqref{eq:tinu}. 
Assume that $k \ge 1$; by the induction hypothesis, 
%
%
\begin{equation} \label{eq:timu1}
\ti{\mu}_{k} =
r_{\xi_{k}} \ti{\mu}_{k-1} = 
r_{\xi_{k}} w_{k-1}\lambda+
 \left(\sum_{u=1}^{p-1} \Bpair{\Lambda}{\wt(\bd_{u})}\right)\delta+
 \left(\sum_{l \in [1,\,k-1]_{\Q}} \Bpair{\Lambda}{\beta_{l}^{\vee}}\right)\delta. 
\end{equation}
If $w_{k-1} \stackrel{\beta_{k}}{\longleftarrow} w_{k}$ is a Bruhat edge, 
then we have $[1,\,k]_{\Q}=[1,\,k-1]_{\Q}$. Also, 
since $\xi_{k}=w_{k}\beta_{k}$, it follows that 
\begin{align*}
r_{\xi_{k}} w_{k-1}\lambda 
  & = w_{k}r_{\beta_{k}}w_{k}^{-1} w_{k-1}\lambda 
    = w_{k}r_{\beta_{k}}w_{k}^{-1} \mcr{ w_{k}r_{\beta_{k}} }\lambda 
    = w_{k}r_{\beta_{k}}w_{k}^{-1} w_{k}r_{\beta_{k}} \lambda 
    = w_{k}\lambda. 
\end{align*}
Therefore, the right-hand side 
of equation \eqref{eq:timu1} is identical to
\begin{align*}
w_{k}\lambda+
 \left(\sum_{u=1}^{p-1} \Bpair{\Lambda}{\wt(\bd_{u})}\right)\delta+
 \left(\sum_{l \in [1,\,k]_{\Q}} \Bpair{\Lambda}{\beta_{l}^{\vee}}\right)\delta. 
\end{align*}
If $w_{k-1} \stackrel{\beta_{k}}{\longleftarrow} w_{k}$ is a quantum edge, 
then we have $[1,\,k]_{\Q}=[1,\,k-1]_{\Q} \cup \bigl\{k\bigr\}$. Also, 
since $\xi_{k}=w_{k}\beta_{k} + \delta$, 
it follows that 
\begin{align*}
r_{\xi_{k}} w_{k-1}\lambda 
 & = r_{ w_{k}\beta_{k} +\delta }w_{k-1}\lambda 
    = r_{ w_{k}\beta_{k} } t_{ w_{k}\beta_{k}^{\vee} } w_{k-1}\lambda
    = \underbrace{ r_{w_{k}\beta_{k}}w_{k-1}\lambda }_{%
         \text{$=w_{k}\lambda$ as above} } - 
      \pair{w_{k-1}\Lambda}{ w_{k}\beta_{k}^{\vee} }\delta \\
 & = w_{k}\lambda - \pair{ \mcr{ w_{k}r_{\beta_{k}} }\Lambda }{ w_{k}\beta_{k}^{\vee} } \delta
   = w_{k}\lambda - \pair{ w_{k}r_{\beta_{k}}\Lambda }{ w_{k}\beta_{k}^{\vee} } \delta \\
 & = w_{k}\lambda + \pair{\Lambda}{\beta_{k}^{\vee}} \delta. 
\end{align*}
Therefore, the right-hand side of equation \eqref{eq:timu1} is identical to
\begin{align*}
& w_{k}\lambda+\pair{\Lambda}{\beta_{k}^{\vee}} \delta + 
 \left(\sum_{u=1}^{p-1} \Bpair{\Lambda}{\wt(\bd_{u})}\right)\delta+
 \left(\sum_{l \in [1,\,k-1]_{\Q}} \Bpair{\Lambda}{\beta_{l}^{\vee}}\right)\delta \\[3mm]
& \hspace*{20mm}
 = w_{k}\lambda+
 \left(\sum_{u=1}^{p-1} \Bpair{\Lambda}{\wt(\bd_{u})}\right)\delta+
 \left(\sum_{l \in [1,\,k]_{\Q}} \Bpair{\Lambda}{\beta_{l}^{\vee}}\right)\delta. 
\end{align*}
This proves equation \eqref{eq:timu}. 
In particular, for $k=n$, we obtain
\begin{align*}
\ti{\mu}_{n} 
 & = r_{\xi_{n}} \cdots r_{\xi_{2}}r_{\xi_{1}}\ti{\nu}_{p} 
   = w_{n}\lambda+
 \left(\sum_{u=1}^{p-1} \Bpair{\Lambda}{\wt(\bd_{u})}\right)\delta+
 \left(\sum_{l \in [1,\,n]_{\Q}} \Bpair{\Lambda}{\beta_{l}^{\vee}}\right)\delta \\[3mm]
 & = x_{p+1}\lambda+ 
 \left(\sum_{u=1}^{p-1} \Bpair{\Lambda}{\wt(\bd_{u})}\right)\delta+
 \Bpair{\Lambda}{\wt(\bd_{p})}\delta \\[3mm]
 & = x_{p+1}\lambda+ 
 \left(\sum_{u=1}^{p} \Bpair{\Lambda}{\wt(\bd_{u})}\right)\delta = \ti{\nu}_{p+1}
\end{align*}
by the definition \eqref{eq:tinu} of $\ti{\nu}_{p+1}$. 
Since $\ti{\nu}_{p} \in W\lambda$ by our induction hypothesis, 
we deduce that $\ti{\nu}_{p+1} \in W\lambda$, as desired. Also, 
by Proposition~\ref{prop:QB}\,(1), we see that for $1 \le p \le s-1$, 
\begin{equation*}
\ti{\nu}_{p}=\ti{\mu}_{0} \gtrdot \ti{\mu}_{1} \gtrdot \ti{\mu}_{2} 
\gtrdot \cdots \gtrdot \ti{\mu}_{n}=\ti{\nu}_{p+1},
\end{equation*}
where $\ti{\mu}_{k}=r_{\xi_{k}}\ti{\mu}_{k-1}$ for $1 \le k \le n$ 
by the definitions. Moreover, since $\bd_{p}$ is a directed $\sigma_{p}$-path, 
it follows from Remark~\ref{rem:QB} that 
the sequence above is a $\sigma_{p}$-chain for 
$(\ti{\nu}_{p},\,\ti{\nu}_{p+1})$. 
Thus we conclude that $\ti{\pi}_{\eta} \in \BB(\lambda)$. 

Because $\ti{\pi}_{\eta} \in \BB(\lambda)$ as shown above, 
and because $\cl(\ti{\pi}_{\eta})=\eta$ and 
$\ti{\nu}_{1}=x_{1}\lambda \in W_{0}\lambda \subset \lambda-Q_{0}^{+}$ 
by the definitions, the element $\ti{\pi}_{\eta}$ 
satisfies conditions (1) and (3) of Proposition~\ref{prop:degree}. 
Therefore, we deduce from Lemma~\ref{lem:C} that $\ti{\pi}_{\eta}(1)$
is of the form:
\begin{equation*}
\ti{\pi}_{\eta}(1) = \lambda-\beta+\bigl(-\Deg (\eta)+L\bigr)\delta
\end{equation*}
for some $\beta \in Q_{0}^{+}$ and $L \in \BZ_{\ge 0}$. 
By Lemma~\ref{lem:C}, in order to prove that $\ti{\pi}_{\eta}=\pi_{\eta}$, 
it suffices to show that $L=0$, or equivalently, 
$-\Deg (\eta)+L \le -\Deg(\eta)$ since $L \in \BZ_{\ge 0}$.
By using \eqref{eq:path}, we see 
from the definition of $\ti{\pi}_{\eta}$ that
%
%
\begin{equation} \label{eq:deg1}
-\Deg (\eta)+L = 
\sum_{p=0}^{s-1} (\sigma_{p+1}-\sigma_{p}) 
\left(\sum_{u=1}^{p} \Bpair{\Lambda}{\wt(\bd_{u})}\right). 
\end{equation}

Now, if we write $\pi_{\eta}$ as
\begin{equation*}
\pi_{\eta}=(\nu_{1},\,\nu_{2},\,\dots,\,\nu_{b}\,;\,
\tau_{0},\,\tau_{1},\,\dots,\,\tau_{b}), 
\end{equation*}
then we have $\nu_{q} \in \lambda-Q_{0}^{+}+K_{q}\delta$ 
for some $K_{q} \in \BZ$, $1 \le q \le b$ (see Remark~\ref{rem:theta}\,(2)); 
observe that $K_{1}=0$ by the definition of $\pi_{\eta}$ 
(see Proposition~\ref{prop:degree}\,(3)), and that 
$0=K_{1} \le K_{2} \le \cdots \le K_{b}$ 
by Remark~\ref{rem:LP}. Since $\cl(\pi_{\eta})=\eta$, 
we deduce that there exist $0 = c_{0} < c_{1} < c_{2} < \cdots < c_{s} = b$ 
such that $\tau_{c_{p}}=\sigma_{p}$ for $0 \le p \le s$, 
and hence $\pi_{\eta}$ can be written as:
\begin{equation*}
\begin{split}
&
(\underbrace{ \nu_{1},\,\dots,\,\nu_{c_{1}} }_{%
  \begin{subarray}{c} \text{mapped to} \\[1mm] \text{$x_{1}\lambda$ by $\cl$} \end{subarray}
 },\ 
\underbrace{ \nu_{c_{1}+1},\,\dots,\,\nu_{c_{2}} }_{%
  \begin{subarray}{c} \text{mapped to} \\[1mm] \text{$x_{2}\lambda$ by $\cl$} \end{subarray}
 },\,\dots,\,
\underbrace{ \nu_{c_{s-1}+1},\,\dots,\,\nu_{c_{s}}=\nu_{b} }_{%
  \text{mapped to $x_{s}\lambda$ by $\cl$}
 } \,;\, \\[3mm]
& \hspace*{20mm}
\underbrace{0=\tau_{0}}_{=\sigma_{0}},\,\tau_{1},\,\dots,\,
\underbrace{\tau_{c_{1}}}_{=\sigma_{1}},\,
\tau_{c_{1}+1},\,\dots,\,
\underbrace{\tau_{c_{2}}}_{=\sigma_{2}},\,\dots,\,
\tau_{c_{s-1}+1},\,\dots,\,
\underbrace{\tau_{c_{s}}=\tau_{b}=1}_{=\sigma_{s}}). 
\end{split}
\end{equation*}
From this, we compute
\begin{align}
-\Deg(\eta) & 
 =\sum_{q=1}^{b}(\tau_{q}-\tau_{q-1})K_{q}
 =\sum_{p=0}^{s-1} \sum_{q=c_{p}+1}^{c_{p+1}} 
  (\tau_{q}-\tau_{q-1})K_{q} \nonumber \\[3mm]
& \ge
 \sum_{p=0}^{s-1} \sum_{q=c_{p}+1}^{c_{p+1}} 
 (\tau_{q}-\tau_{q-1})K_{c_{p}+1} 
 \quad \text{since $K_{q} \ge K_{c_{p}+1}$ 
  for all $c_{p}+1 \le q \le c_{p+1}$} \nonumber \\[3mm]
& 
 = \sum_{p=0}^{s-1} (\tau_{c_{p+1}}-\tau_{c_{p}}) K_{c_{p}+1}
 = \sum_{p=0}^{s-1} (\sigma_{p+1}-\sigma_{p}) K_{c_{p}+1}. \label{eq:deg2}
\end{align}
Therefore, by \eqref{eq:deg1} and \eqref{eq:deg2}, 
in order to show the inequality $-\Deg (\eta)+L \le -\Deg(\eta)$, 
it suffices to show that
%
%
\begin{equation} \label{eq:deg3}
K_{c_{p}+1} \ge \sum_{u=1}^{p} \Bpair{\Lambda}{\wt(\bd_{u})} \quad 
\text{for all $0 \le p \le s-1$}. 
\end{equation}
We show this inequality by induction on $p$. 
If $p=0$, then the assertion is obvious 
since $K_{c_{p}+1}=K_{1}=0$ as seen above. 
Assume that $s-1 \ge p > 0$. 
Take $\mu_{0},\,\mu_{1},\,\dots,\,\mu_{m} \in W\lambda$ such that 
\begin{equation*}
\nu_{c_{p}} = \mu_{0} \gtrdot \mu_{1} \gtrdot \cdots \gtrdot \mu_{m}=\nu_{c_{p}+1}
\end{equation*}
(for example, take a $\tau_{c_{p}}$-chain for $(\nu_{c_{p}},\,\nu_{c_{p}+1})$), 
and let $\zeta_{k} \in \prr$ be the positive real root such that 
$\mu_{k}=r_{\zeta_{k}}\mu_{k-1}$, $1 \le k \le m$. 
For each $0 \le k \le m$, let $v_{k} \in W_{0}^{J}$ be a unique element 
in $W_{0}^{J}$ such that $\cl(\mu_{k})=v_{k}\Lambda$; 
remark that $v_{0}=x_{p}$ and $v_{m}=x_{p+1}$. 
By repeated application of Proposition~\ref{prop:QB}\,(2), 
we obtain a directed path (not shortest in general)
\begin{equation*}
\bd : 
x_{p}=v_{0} \stackrel{\gamma_{1}}{\longleftarrow} 
v_{1} \stackrel{\gamma_{2}}{\longleftarrow} 
v_{2} \stackrel{\gamma_{3}}{\longleftarrow} \cdots 
\stackrel{\gamma_{m}}{\longleftarrow} v_{m}=x_{p+1}
\end{equation*}
from $x_{p+1}$ to $x_{p}$ in the parabolic quantum Bruhat graph, 
where $\gamma_{k} \in \Delta_{0}^{+} \setminus \Delta_{0,J}^{+}$ 
for $1 \le k \le m$ are defined by
\begin{equation*}
\gamma_{k}:=
 \begin{cases}
 v_{k}^{-1}\zeta_{k} & \text{if $\zeta_{k} \in \Delta_{0}^{+}$}, \\[1.5mm]
 v_{k}^{-1}(\zeta_{k}-\delta)
     & \text{if $\zeta_{k} \in 
     \bigl\{-\gamma+\delta \mid \gamma \in \Delta_{0}^{+}\bigr\}$};
 \end{cases}
\end{equation*}
recall that 
$v_{k-1} \stackrel{\gamma_{k}}{\longleftarrow} v_{k}$ 
is a Bruhat edge if and only if $\zeta_{k} \in \Delta_{0}^{+}$. 
By the same argument as for equation \eqref{eq:timu},  
we can show that for $0 \le k \le m$,
\begin{equation*}
\mu_{k}=
r_{\zeta_{k}} \cdots r_{\zeta_{2}}r_{\zeta_{1}}\nu_{c_{p}}=
v_{k}\lambda+K_{c_{p}}\delta + 
 \left(\sum_{l} \Bpair{\Lambda}{\gamma_{l}^{\vee}}\right)\delta,
\end{equation*}
where the summation above is over all $1 \le l \le k$ 
for which $v_{l-1}=\mcr{v_{l}r_{\gamma_{l}}} 
   \stackrel{\gamma_{l}}{\longleftarrow} v_{l}$ is a quantum edge. 
In particular, for $k=m$, we obtain
\begin{equation*}
\nu_{c_{p}+1} = \mu_{m}= x_{p+1}\lambda+K_{c_{p}}\delta + 
\pair{\Lambda}{\wt (\bd)}\delta, 
\end{equation*}
and hence $K_{c_{p}+1} = K_{c_{p}}+\pair{\Lambda}{\wt (\bd)}$. 
Here we see from \cite[Proposition~8.1]{LNSSS1} that 
$\pair{\Lambda}{\wt (\bd)} \ge \pair{\Lambda}{\wt (\bd_{p})}$. 
Also, by the induction hypothesis (note that $c_{p-1} < c_{p}$), 
\begin{equation*}
K_{c_{p}} \ge K_{c_{p-1}+1} \ge 
\sum_{u=1}^{p-1} \Bpair{\Lambda}{\wt(\bd_{u})}.
\end{equation*}
Combining these, we obtain 
\begin{equation*}
K_{c_{p}+1} = K_{c_{p}}+\pair{\Lambda}{\wt (\bd)} 
\ge \sum_{u=1}^{p-1} \Bpair{\Lambda}{\wt(\bd_{u})} + 
\pair{\Lambda}{\wt (\bd_{p})}=
\sum_{u=1}^{p} \Bpair{\Lambda}{\wt(\bd_{u})}.
\end{equation*}
Thus, we have proved the inequality 
$-\Deg (\eta)+L \le -\Deg(\eta)$, and hence 
the equality $\ti{\pi}_{\eta}=\pi_{\eta}$, as desired. 

Finally, from equation \eqref{eq:deg1} together with $L=0$ shown above, 
we deduce that 
\begin{align*}
-\Deg (\eta) & = 
\sum_{p=0}^{s-1} (\sigma_{p+1}-\sigma_{p}) 
\left(\sum_{u=1}^{p} \Bpair{\Lambda}{\wt(\bd_{u})}\right)
= 
\sum_{p=1}^{s-1}\sum_{u=1}^{p} (\sigma_{p+1}-\sigma_{p}) 
\Bpair{\Lambda}{\wt(\bd_{u})} \\[3mm]
& = 
\sum_{p=1}^{s-1} \left\{\sum_{q=p}^{s-1} (\sigma_{q+1}-\sigma_{q})\right\}
\Bpair{\Lambda}{\wt(\bd_{p})} 
= \sum_{p=1}^{s-1} (\sigma_{s}-\sigma_{p})
\Bpair{\Lambda}{\wt(\bd_{p})} \\[3mm]
& = \sum_{p=1}^{s-1} (1-\sigma_{p})
\Bpair{\Lambda}{\wt(\bd_{p})}. 
\end{align*}
Thus we have proved formula \eqref{eq:deg}. 
This completes the proof of Theorem~\ref{thm:main}. \qed


{\small
\setlength{\baselineskip}{13pt}
\renewcommand{\refname}{References}

}

\end{document}